\documentclass[fleqn
,12pt
]{article}
\usepackage{amsmath,amssymb,amsthm,xcolor,framed,esint,dsfont}
\usepackage[english]{babel}
\usepackage[margin=3.5cm
,top=3cm,bottom=3cm
]{geometry}
\usepackage[title]{appendix}
\usepackage[bookmarks]{hyperref}
\numberwithin{equation}{section}

\def\Xint#1{\mathchoice
	{\XXint\displaystyle\textstyle{#1}}%
	{\XXint\textstyle\scriptstyle{#1}}%
	{\XXint\scriptstyle\scriptscriptstyle{#1}}%
	{\XXint\scriptscriptstyle\scriptscriptstyle{#1}}%
	\!\int}

\def\XXint#1#2#3{{\setbox0=\hbox{$#1{#2#3}{\int}$}
		\vcenter{\hbox{$#2#3$}}\kern-.5\wd0}}
         
\newcommand{\R}{{\mathbb R}}
\newcommand{\Z}{{\mathbb Z}}
\newcommand{\T}{{\mathbb T}}

\newcommand{\e}{{\epsilon}}
\newcommand{\loc}{{\mathrm{loc}}}
\newcommand{\I}{{\mathrm{id}}}

\DeclareMathOperator{\Glim}{\Gamma-\lim}
\DeclareMathOperator{\Glimsup}{\Gamma-\limsup}

\DeclareMathOperator*{\esssup}{ess\,sup}
\DeclareMathOperator*{\esslim}{ess\,lim}
\DeclareMathOperator{\sign}{sign}
\DeclareMathOperator{\dist}{dist}
\DeclareMathOperator{\supp}{supp}
\DeclareMathOperator{\dv}{div}
\DeclareMathOperator{\tr}{tr}
\DeclareMathOperator{\curl}{curl}
\DeclareMathOperator{\re}{{\mathfrak Re}}

\newtheorem{theorem}{Theorem}[section]
\newtheorem{corollary}[theorem]{Corollary}
\newtheorem{lemma}[theorem]{Lemma}
\newtheorem{proposition}[theorem]{Proposition}
\theoremstyle{definition}

\newtheorem{remark}[theorem]{Remark}

\title{Compensation effects for anisotropic energies of 
two-dimensional unit vector fields}

\date{
\today
}
\author{Lia Bronsard\footnote{Department
of Mathematics and Statistics, McMaster University,
 \ttfamily{bronsard@mcmaster.ca}}
 \and 
 Dmitry Golovaty\footnote{Department of Mathematics, The University of Akron, Akron OH 44325,
\ttfamily{dmitry@uakron.edu}}
  \and 
  Xavier Lamy \footnote{Institut de Mathématiques de Toulouse
  UMR5219, Université de Toulouse, CNRS, 
  F-31062 Toulouse Cedex 9, 
  \ttfamily{xavier.lamy@math.univ-toulouse.fr}  }
  \footnote{Institut Universitaire de France (IUF)}
  \and 
  Peter Sternberg\footnote{Department of Mathematics, Indiana University, Bloomington, IN 47405, \ttfamily{sternber@iu.edu}}
  }

\begin{document}

\maketitle

\begin{abstract}
We study the
highly anisotropic energy of 
two-dimensional unit vector fields given by
\begin{align*}
E_\epsilon(u)=
\int_{\Omega}
(\mathrm{div}\,u)^2 +
\epsilon(\mathrm{curl}\,u)^2\, dx\,,
\quad u\colon\Omega\subset\R^2\to\mathbb S^1\,
\end{align*}
 in the limit $\e\to 0$.
This energy clearly loses control on the full gradient of $u$ as $\e\to 0$,
but,
adapting tools from hyperbolic conservations laws, we show that it still controls derivatives of order 1/2.
In particular,
any bounded energy sequence $E_\e(u_\e)\leq C$
is compact in $W^{s,3}_{\loc}(\Omega)$ for $s<1/2$.
Moreover, this order 1/2 of differentiability  is optimal, in the sense that
any map $u\in W^{1/2,4}(\Omega;\mathbb S^1)$ 
is a limit of a bounded energy sequence.
We also establish compactness of
 boundary traces in $L^1(\partial\Omega)$,
and characterize the $\Gamma$-limit 
 in the simpler case of maps of a single variable and in the case of a thin-film model.
\end{abstract}

\section{Introduction}

We consider a smooth bounded open set $\Omega\subset\R^2$,
and, for $0<\e\ll 1$, 
the highly anisotropic energy functional
\begin{align}\label{eq:Eeps}
E_\e(u)=\int_\Omega (\dv u)^2+\e(\curl u)^2\, dx,\qquad
\text{for } u\in H^1(\Omega;\mathbb S^1)\,.
\end{align}
This is a two-dimensional version of the Oseen-Frank energy for nematic liquid crystals, 
in a regime of very disparate elastic constants \cite{Virga94}.
From the point of view of modelling, 
one of the most relevant questions is to determine the behavior of minimizers (subject to some boundary constraint)
as $\e\to 0$, or, more generally equicoercivity properties of $E_\e$ and its $\Gamma$-limit as $\e\to 0$, 
see \cite[Definition~1.5]{braides}.
Perhaps surprisingly, this seems to be a quite subtle task, 
and we only provide partial answers.

The mathematical study of liquid crystal models 
often relies on the one-constant approximation,
where the elastic part of the energy is of the form $\int |\nabla u|^2\, dx$, 
even though actual liquid crystals have anisotropic elastic properties \cite{degennes}.
It is well-known that this is not just a cosmetic simplification: anisotropic energies present significant mathematical challenges 
with respect to the isotropic case
(see e.g. \cite{HKL86,HKL88} for anisotropic harmonic maps and \cite{CKP13,GMS24,BCS25,BCSvB} for anisotropic Ginzburg-Landau energies).
The effect of very disparate elastic constants has been analyzed in \cite{GSV19,GNSV20,GNS21} 
for models of Ginzburg-Landau type which relax the $\mathbb S^1$-valued constraint,
see also \cite{BMMPS}.

Here we focus on the interplay between the degenerate ellipticity $\e\to 0$ 
and the pointwise constraint $u(x)\in\mathbb S^1$,
and demonstrate that this interplay gives rise to interesting compensation phenomena.
Similar compensations have been observed 
in \cite{IO08}, 
where a bound on the divergence of smooth $\mathbb S^1$-valued vector fields with tangential boundary conditions is enough to provide strong compactness,
as well as in \cite{RS03} where the divergence vanishes in the limit, 
and in the Aviles-Giga problem \cite{ADM99,DKMO01} where the divergence is zero but the $\mathbb S^1$ constraint is relaxed.
This is also reminiscent of the compensated integrability established in \cite{GRS24}
where degenerate ellipticity is coupled with a cone-valued constraint.

Naively, if $u_\e\colon\Omega\to\mathbb S^1$ minimizes the energy $E_\e$ given by \eqref{eq:Eeps} with respect to some boundary conditions,
as $\e\to 0$
 one could expect  a limit map $u_*\colon\Omega\to\mathbb S^1$ to minimize
\begin{align}\label{eq:E0}
E_0(u)=\int_\Omega (\dv u)^2\, dx\,,
\quad
\text{for } u\colon\Omega\to\mathbb S^1\,
\text{ s.t. } \dv u\in L^2(\Omega)\,.
\end{align}
Here $\dv u$ is the distributional divergence of a vector field $u$.
This already raises simple questions which do not have obvious answers:
Is there enough compactness to obtain a limit $u_*$ with values into $\mathbb S^1$ ?
Does $u_*$ still satisfy the same boundary data?
How singular can $u_*$ be ?

With respect to this last question,
it is instructive to consider
maps $u\colon \Omega\to\mathbb S^1$ which are divergence-free : they satisfy $E_0(u)=0$,
hence they are minimizers of \eqref{eq:E0}.
Such maps can be extremely wild: any weak solution of the eikonal equation $|\nabla \varphi|= 1$ a.e. in $\Omega$ provides a divergence-free map $u\colon\Omega\to \mathbb S^1$ given by $u=\nabla^\perp\varphi$,
and general weak solutions of the eikonal equation are very flexible, for instance they can be uniformly close to any $\psi\colon\Omega\to\R$ with $|\nabla\psi|\leq 1$ \cite{DM97}.
Instructed by the compactness results of \cite{ADM99,DKMO01,RS03,IO08}, we may expect that such pathological behavior should be ruled out for limits of bounded energy sequences.
But it is still natural to wonder about two  explicit
 examples of elementary divergence-free singularities.
The first example is that of a
 divergence-free jump
\begin{align}\label{eq:u*jump}
u^{\mathrm{jump}}(x)=u^- \mathbf 1_{x_1<0}
+
u^+ \mathbf 1_{x_1>0},
\qquad 
&
u^+\neq u^-\in\mathbb S^1, 
\\
\quad 
&
(u^+-u^-)\cdot \mathbf e_1 =0,
\nonumber
\end{align}
where $x=(x_1,x_2)$. The second example is that of a divergence-free vortex
\begin{align}\label{eq:u*vortex}
u^{\mathrm{vort}}(x)=\frac{ix}{|x|}=\frac{(-x_2,x_1)}{|x|}.
\end{align}
 Can these be obtained as  limits of minimizers $u_\e$ of $E_\e$ ?
 Do singularities carry an energy cost ?

All these questions are related to 
 determining the $\Gamma$-limit of $E_\e$
 (say, with respect to convergence in the sense of distributions),
 possibly under some restriction on boundary values.
In that respect, a first step is to understand the class of maps with finite $\Glimsup$, that is,
 which maps can be obtained as limits of bounded energy sequences:
\begin{align}\label{eq:Glimsupfinite}
\mathcal S
&
:=
\left\lbrace \Glimsup_{\e\to 0} E_\e <\infty \right\rbrace
\\
&
=
\Big\lbrace u\in L^\infty(\Omega;\R^2)
\colon 
\exists \e_k\searrow 0,
\,
u_{\e_k}\in H^1(\Omega;\mathbb S^1)
\nonumber
\\
&
\hspace{7.3em}
\text{s.t. }
\limsup_{k\to \infty} E_{\e_k}(u_{\e_k}) <\infty
%\nonumber
%\\
%&
%\hspace{8.6em}
\text{ and }
u_{\e_k}\to u\text{ in }\mathcal D'(\Omega)
\Big\rbrace
\,.
\nonumber
\end{align}
Our main results, to be presented in more detail below, imply
that the $\mathbb S^1$ constraint is preserved in the limit,
as well as boundary data. 
Concerning singularities, we find that
\begin{align*}
 u^{\mathrm{jump}},\;u^{\mathrm{vort}}\notin \mathcal S\,.
\end{align*}
In other words, these singularities cannot be approximated with finite $E_\e$ energy.
Further,
 we establish that
although the functional $E_\e$ in \eqref{eq:Eeps} gives up any control on $\curl u$ as $\e\to 0$,
the $L^2$ control on $\dv u$ coupled with the $\mathbb S^1$ constraint is enough 
to control roughly ``half a derivative'', in the sense of fractional Sobolev $W^{1/2,p}$ regularity:
\begin{align*}
%W^{\frac 12 ,4}(\Omega;\mathbb S^1)\subset 
\mathcal S \subset W^{\frac 12 -, 3}_{\loc}(\Omega;\mathbb S^1)\,.
\end{align*}
This half order of differentiation turns out to be optimal, 
in the sense that any map with half a gradient (but higher integrability)
can be approximated by bounded energy sequences:
\begin{align*}
W^{\frac 12 ,4}(\Omega;\mathbb S^1)\subset 
\mathcal S\,.
\end{align*}
Finally, in the simpler cases of a one-dimensional domain and a thin domain, we are able to completely determine the $\Gamma$-limit.
Next we proceed to describe our results more precisely.

\subsection{Compensated regularity: interior estimate}\label{ss:intro_comp}

The first compensation phenomenon we observe is that
energy boundedness provides
$\e$-independent regularity estimates
with a fractional order of derivatives. We begin with a brief review of the relevant function spaces.

\subsubsection{Fractional Sobolev and Besov spaces}

In a smooth bounded open set $\Omega\subset\R^d$,
for a differentiability order $0<s<1$ 
and integrability exponent $p\geq 1$,
fractional differentiability of a function $u\in L^p(\Omega)$ can be measured in terms of the difference quotient
\begin{align*}
\frac{\|D^h u\|_{L^p(\Omega\cap(\Omega-h))}}{|h|^s}\,,\qquad h\in\R^d\setminus \lbrace 0\rbrace\,, 
\end{align*}
where $D^h$ denotes the finite difference operator
\begin{align}\label{eq:Dh}
D^h u(x)=u(x+h)-u(x)\,.
\end{align}
One can consider, for instance,
 the
 fractional Sobolev space $W^{s,p}(\Omega)$,
 which consists of functions $u\in L^p(\Omega)$ such that
\begin{align}
\label{eq:Wsp}
|u|_{W^{s,p}(\Omega)}^p
&
=\int_{\R^d} 
\bigg(
\frac{\|D^h u\|_{L^p(\Omega\cap (\Omega -h))}}{|h|^{s}}
\bigg)^p\, \frac{dh}{|h|^d} <\infty\,.
\end{align}
Note that this coincides with the commonly used Gagliardo seminorm \cite{DNPV12}
\begin{align*}
|u|_{W^{s,p}(\Omega)}^p
=\int_{\Omega\times\Omega} \frac{|u(x)-u(y)|^p}{|x-y|^{sp}}\, \frac{dxdy}{|x-y|^d}\,.
\end{align*}
We will also consider the Besov space $B^{s}_{p,\infty}(\Omega)$
 of functions $u\in L^p(\Omega)$ such that
\begin{align}
 |u|_{B^s_{p,\infty}(\Omega)}
 &
 =\esssup_{h\in\R^d} 
 \frac{\|D^h u\|_{L^p(\Omega\cap (\Omega -h))}}{|h|^{s}} <\infty \,.
 \label{eq:Bsp}
 \end{align}
Note that we have the inclusion
\begin{align*}
B^{s}_{p,\infty}(\Omega)\subset W^{s',p}(\Omega)\qquad\text{for }0 < s' < s < 1,\; p\geq 1\,,
\end{align*}
since
\begin{align*}
|u|_{W^{s',p}(\Omega)}^p
&
=\int_{\R^2} 
\bigg(
\frac{\|D^h u\|_{L^p(\Omega\cap (\Omega -h))}}{|h|^{s'}}
\bigg)^p\, 
\frac{dh}{|h|^{d}}
\\
&
\leq 
|u|_{B^s_{p,\infty}}^p \int_{|h|\leq 1}\frac{dh}{|h|^{d-(s-s')p}}
+2^p\|u\|_{L^p}^p\int_{|h|\geq 1}\frac{dh}{|h|^{d+s'p}}\,.
\end{align*}
Moreover, this embedding is compact, as follows e.g. from \cite[Remark~1, p.233]{triebel83}.

\subsubsection{Interior estimate and compactness}

We prove that the energy controls derivatives of order $1/2$ in $L^3$,
in the sense of the Besov space $B^{1/2}_{3,\infty}$
(and therefore also in $W^{s,3}$ for all $s<1/2$).

\begin{theorem}\label{t:reg_estim}
For any bounded open sets $\Omega'\subset\subset\Omega\subset\R^2$, there exists $C=C(\Omega',\Omega)>0$ such that
\begin{align}\label{eq:reg_estim}
\|D^h u\|_{L^3(\Omega'\cap(\Omega'-h))}
&
\leq C \left(1+\|\dv u\|_{L^2(\Omega)}\right)^{\frac 12} |h|^{\frac 12}\qquad\forall h\in\R^2
\,,
\end{align}
for all $u\in H^1(\Omega;\mathbb S^1)$,
and more generally for all $u$ in the finite $\Glimsup$ set $\mathcal S$ defined in \eqref{eq:Glimsupfinite}.
\end{theorem}

The proof of Theorem~\ref{t:reg_estim} relies on 
notions of entropy productions and a kinetic formulation
adapted from the theory of conservation laws, see \textsection\ref{s:comp_reg}.

By compactness of the embedding $B^{1/2}_{3,\infty}\subset W^{s,3}$ for all $s\in (0,1/2)$,
the compensated regularity estimate
\eqref{eq:reg_estim} 
directly
implies the following strong compactness properties 
for sequences with bounded energy.

\begin{corollary}\label{c:Ws3}
If $(u_\e)\subset H^1(\Omega;\mathbb S^1)$
satisfies
$\limsup_{\e\to 0}E_\e(u_\e)<\infty$,
there exists
\begin{align*}
u\in 
B^{1/2}_{3,\infty,\loc}(\Omega;\mathbb S^1)
%=  \bigcap_{\Omega'\subset\subset\Omega}
%B^{1/2}_{3,\infty}(\Omega';\mathbb S^1)
%\,,
%\quad
\text{ with }
\dv u\in L^2(\Omega)\,,
\end{align*}
and
 a sequence $\e_k\to 0$ such that
\begin{align*}
u_{\e_k} \longrightarrow u\qquad\text{strongly in }W^{s,3}(\Omega';\R^2)\,,
\end{align*}
for all $0\leq s < 1/2$ and $\Omega'\subset\subset\Omega$.
%In particular, the finite $\Glimsup$ set $\mathcal S$ defined in \eqref{eq:Glimsupfinite} is contained in $B^{1/2}_{3,\infty,\loc}(\Omega;\mathbb S^1)$.
\end{corollary}

Let us emphasize that the estimate \eqref{eq:reg_estim} 
is false if we only assume that $u\in H_{\dv}(\Omega;\mathbb S^1)$, that is, $u\in L^2(\Omega;\mathbb S^1)$ with square-integrable distributional divergence $\dv u\in L^2(\Omega)$.
For instance, the pure jump $u^{\mathrm{jump}}$ given in \eqref{eq:u*jump} belongs to 
$H_{\dv}(B_1;\mathbb S^1)$ since $\dv u^{\mathrm{jump}}=0$, 
but not to $B^{1/2}_{3,\infty,\loc}(B_1)$ since $\|D^h u^{\mathrm{jump}}\|_{L^3(B_{1/2})}\sim c |u^+-u^-||h_1|^{1/3}$ as $|h|=|(h_1,h_2)|\to 0$, for some $c>0$.
Hence Theorem~\ref{t:reg_estim} shows in particular that $u^{\mathrm{jump}}$ does not belong to the finite $\Glimsup$ set $\mathcal S$ for $\Omega=B_1$.

\begin{remark}
If $\Omega\subset\R^2$ is simply connected, maps $u\in H^1(\Omega;\mathbb S^1)$ can be lifted to $\varphi\in H^1(\Omega;\R)$ such that $u=e^{i\varphi}$, see e.g. \cite[\textsection 1]{BM21}. In terms of this phase $\varphi$, we have
\begin{align*}
E_\e(u) 
&
=\int_\Omega
\langle A_\e(\varphi)\nabla\varphi , \nabla\varphi \rangle \, dx\,,
\quad 
A_\e (\varphi)
=ie^{i\varphi}\otimes ie^{i\varphi} +\e\, e^{i\varphi}\otimes e^{i\varphi}\in\R^{2\times 2}\,.
\end{align*}
The degenerate ellipticity is seen in the fact that the nonnegative symmetric matrix $A_\e(\varphi)$ has an eigenvalue $\e\to 0$,
in the degenerate eigendirection $e^{i\varphi}$.
If instead we had an energy of the form
\begin{align*}
\widetilde E_\e(\varphi) =\int_\Omega
\langle \tilde A_\e\nabla\varphi , \nabla\varphi \rangle \, dx\,,
\quad 
\tilde A_\e =iv\otimes iv +\e\, v\otimes v\in\R^{2\times 2}\,,
\end{align*}
for some constant vector $v\in\mathbb S^1$, then any integrable function $x\mapsto f(\langle v, x\rangle)$ could arise as a bounded energy limit, simply by taking a smooth approximation $f_j\to f$ and adjusting $\e_j\to 0$.
In particular there would be no regularity estimate for bounded energy limits: from that perspective, the compensation effect shown by Theorem~\ref{t:reg_estim} seems related to the specific  nonlinear dependence of the degenerate direction $e^{i\varphi}$ on $\varphi$.
\end{remark}

\subsubsection{Consequence on the distributional Jacobian}

Above we have  seen that Theorem~\ref{t:reg_estim} implies that the pure jump $u^{\mathrm{jump}}$ given by \eqref{eq:u*jump} is not a bounded energy limit in $\Omega=B_1$, 
because it does not belong to $B^{1/2}_{3,\infty,\loc}(B_1)$.
We cannot invoke the same reasoning to rule out the vortex $u^{\mathrm{vort}}$ given in \eqref{eq:u*vortex} as a possible  bounded energy limit,
because $u^{\mathrm{vort}}$
belongs to 
 $B^{2/3}_{3,\infty}(B_1)\subset B^{1/2}_{3,\infty}(B_1)$ as a consequence of the inequalities  
\begin{align*}
\|D^h u^{\mathrm{vort}}\|_{L^3(B_1)}^3 
&
\leq \int_{2|h|<|x| < 1} |D^h u^{\mathrm{vort}}|^3\, dx  +\int_{|x|<2|h|}|D^h u^{\mathrm{vort}}|^3\, dx 
\\
&
\leq \int_{2|h|<|x|<1}\frac{|h|^3}{(|x|/2)^3}\, dx  +\int_{|x|<2|h|}8\, dx
 \lesssim |h|^2\,.
\end{align*}
However, from Theorem~\ref{t:reg_estim} it does nonetheless follow that $u^{\mathrm{vort}}$ cannot be a bounded energy limit  in $B_1$,
but for a subtler reason: 
the existence of a continuous notion of Jacobian on $W^{1/3,3}(B_1;\mathbb S^1)$, see \cite[Corollary~8.1]{BM21},
which allows one to pass to the limit in the identity $Ju_\e=\det(\nabla u_\e)=0$ satisfied by $u_\e\in H^1(\Omega;\mathbb S^1)$.

Recall that a map $u\in W^{1,1} \cap L^\infty (\Omega;\R^2)$, has a well-defined distributional Jacobian
\begin{align*}
Ju=\frac 12 \curl (u\wedge \nabla u)\,,
\end{align*}
see \cite[\textsection 1.2]{BM21},
 which coincides with $Ju=\det(\nabla u)$ if $u\in W^{1,2}(\Omega;\R^2)$.
In particular we have $Ju=0$ for $u\in W^{1,2}(\Omega;\mathbb S^1)$ because $\nabla u$ has rank one.
And the vortex satisfies
 $Ju^{\mathrm{vort}}=\pi\delta_0$ \cite[Remark~1.9]{BM21}.

For $\mathbb S^1$-valued maps, 
the distributional Jacobian can be extended to the space
$W^{1/p,p}(\Omega;\mathbb S^1)$ for any $p> 1$: 
there exists a continuous map $J\colon W^{1/p,p}(\Omega;\mathbb S^1)\to\mathcal D'(\Omega)$ 
which coincides with the usual Jacobian on $W^{1/p,p}\cap W^{1,1}(\Omega;\mathbb S^1)$, see  \cite[\textsection 8.1]{BM21}.
In particular, since $\mathcal S\subset W^{1/3,3}_{\loc}(\Omega;\mathbb S^1)$ thanks to Theorem~\ref{t:reg_estim},
any map $u\in\mathcal S$ has a well-defined Jacobian $Ju$.

Recall also that maps $v\in H^1(\Omega;\mathbb S^1)$ have a well-defined winding number $\deg(v;\Gamma)\in \Z$ on any smooth curve $\Gamma\subset\Omega$ homeomorphic to a circle, 
which depends only on the homotopy class of $\Gamma$,
see e.g. \cite[\textsection 12]{BM21}. This leads us to the following Corollary whose proof is detailed in \textsection\ref{ss:csq_reg_kin}.

\begin{corollary}\label{c:jac}
If $u$ belongs to the finite $\Glimsup$ set $\mathcal S$ defined in \eqref{eq:Glimsupfinite},
then
 we have that $Ju=0$.
Moreover, $u$ has a well-defined 
winding number $\deg(u;\Gamma)\in\mathbb Z$ along any smooth curve $\Gamma\subset\Omega$ homeomorphic to a circle,
such that 
$\deg(u;\Gamma)=\lim \deg(u_\e;\Gamma)$ for any sequence $u_\e\to u$ in $\mathcal D'(\Omega)$ with $E_\e(u_\e)\leq C$.
\end{corollary}

In particular, the vortex $u^{\mathrm{vort}}$ defined in \eqref{eq:u*vortex}, which satisfies $Ju^{\mathrm{vort}}=\delta_0\neq 0$, does not belong to $\mathcal S$ in $\Omega=B_1$.

\subsection{Convergence of traces}

The interior estimate of Theorem~\ref{t:reg_estim} does not provide any control on boundary data.
For vector fields $u\in H^1(\Omega;\mathbb S^1)$ which are tangent at the boundary ($u\cdot\nu=0$ on $\partial\Omega$, where $\nu$ is a unit normal on $\partial\Omega$) and sufficiently smooth domains $\Omega$,
an estimate up to the boundary,
\begin{align*}
|u|^2_{B^{1/2}_{3,\infty}(\Omega)}\lesssim 1+\|\dv u\|_{L^2(\Omega)}\,,
\end{align*}
can actually be inferred from Theorem~\ref{t:reg_estim} via a reflection argument coupled with the use of conformal coordinates (to flatten the boundary while keeping control on the divergence).
See \cite{BMMPS} for related work.
But for general maps $u\in H^1(\Omega;\mathbb S^1)$ 
 it is not clear to us whether such a global estimate is valid.
Nevertheless, we are able to show compactness of traces for bounded energy sequences.

\begin{theorem}\label{t:trace}
Let $\Omega\subset\R^2$ be a bounded $C^2$ open set.
For any family  of maps $u_\e\colon\Omega\to\R^2$ such that 
$\limsup_{\e\to 0}E_\e(u_\e)<\infty$,
the limit $u_*$ provided by Corollary~\ref{c:Ws3}
has a strong $L^1$ trace on $\partial\Omega$,
and $\tr(u_{\e_k})\to\tr(u_*)$ strongly in $L^p(\partial\Omega)$ for all $1\leq p<\infty$.
\end{theorem}

Note that vector fields with $L^2$ divergence
always have a weak normal trace which ensures the validity of the Gauss-Green formula, see e.g. \cite{CF99}, 
but this weak normal trace may not be attained in any pointwise sense \cite[Example~2.7]{LS22},
and tangential components may  not even have a weak trace \cite[Remark~2.4]{CF99}. 
Here we really take advantage of the fact that $u_*\in\mathcal S$ to define its strong trace, see Proposition~\ref{p:tracelim}.
As for Theorem~\ref{t:reg_estim}, the proof of Theorem~\ref{t:trace}
relies on ideas from the theory of conservation laws,
 in particular from \cite{CF99} and \cite{vasseur01}.

\subsection{Recovery sequences}\label{ss:intro_recov}

Thanks to Theorem~\ref{t:reg_estim}, if  a map $u\colon\Omega\to\mathbb S^1$ is a limit of maps $u_\e$ with bounded energy, 
then it must have
a certain fractional regularity of order 1/2, namely
 $B^{1/2}_{3,\infty,\loc}(\Omega)$.
 % some minimal fractional regularity.
Conversely,
it is natural to ask which sufficient conditions on $u$ ensure that it can be approximated  with bounded energy.
In other words, we would like to identify large subsets of the finite $\Glimsup$ set $\mathcal S$ defined in 
\eqref{eq:Glimsupfinite}.
We are only able to obtain an incomplete converse to
 Theorem~\ref{t:reg_estim},
which does however confirm the optimality of the order $1/2$ of differentiability: 
any map in $W^{1/2,4}(\Omega)$, see  \eqref{eq:Wsp}, satisfies 
the condition \eqref{eq:suffcondrecov} and, 
according to our next main result, 
belongs therefore to $\mathcal S$.

\begin{theorem}\label{t:recovery2D}
Let $\Omega\subset\R^2$ be a $C^3$
open set,
and $u\colon\Omega\to\mathbb S^1$ possess square-integrable distributional divergence $\dv u\in L^2(\Omega)$, and such that 
\begin{align}\label{eq:suffcondrecov}
 \Xint{-}_{B_\delta} \bigg(\frac{\|D^h u\|_{L^4(\Omega\cap(\Omega-h))}}{\delta^{1/2}}\bigg)^4 \, dh \longrightarrow 0
\qquad\text{ as }\delta\to 0\,,
\end{align}
where $D^h$ is the finite difference operator defined in \eqref{eq:Dh}
and $\Xint{-}$ denotes integral average. 
Then there exists $(u_\e)\subset H^1(\Omega;\mathbb S^1)$ such that 
\begin{align*}
u_\e\to u\text{ in }L^2(\Omega;\R^2)
\quad\text{and }
E_\e(u_\e)\to\int_\Omega (\dv u)^2\, dx,
\end{align*}
as $\e\to 0$.
\end{theorem}

The recovery sequence $u_\e$ 
will be constructed by convolving $u$ with a smoothing kernel
and projecting it onto $\mathbb S^1$.
Convolution commutes with taking the divergence, 
but projection onto $\mathbb S^1$ does not:
the main step in the proof  of Theorem~\ref{t:recovery2D} is to estimate the extra error thus introduced,
 by adapting commutator estimates from \cite{CET94,DLI15}.

From Theorems~\ref{t:reg_estim} and \ref{t:recovery2D},
we infer that the finite $\Glimsup$ set $\mathcal{S}$ given by
\eqref{eq:Glimsupfinite}  satisfies the inclusions
\begin{align*}
W^{\frac 12,4}(\Omega;\mathbb S^1) \subset 
\mathcal S
\subset
B^{\frac 12}_{3,\infty,\loc}(\Omega;\mathbb S^1)\subset W_{\loc}^{\frac 12 -,3}(\Omega;\mathbb S^1)\,.
\end{align*}
In fact, 
for
$u\in W^{\frac 12,4}(\Omega;\mathbb S^1)$
or more generally $u$ satisfying \eqref{eq:suffcondrecov}, the $\Gamma$-limit of $E_\e$ at $u$ 
(see e.g. \cite[Definition~1.5]{braides})
 is completely determined, 
\begin{align*}
\Big(\Glim_{\e\to 0} E_\e \Big) (u) =\int_\Omega (\dv u)^2\, dx\,.
\end{align*}
Indeed,  the upper bound is covered by Theorem~\ref{t:recovery2D},
and the lower bound follows from 
the lower semicontinuity of 
$\int (\dv u_\e)^2\, dx$ 
with respect to weak $L^2$ convergence of $\dv u_\e$.
This leaves the determination of the $\Gamma$-limit open 
in the set of 
 maps $u\in B_{3,\infty,\loc}^{1/2}(\Omega;\mathbb S^1)$ which, recalling Corollary~\ref{c:jac}, satisfy $Ju=0$, but not \eqref{eq:suffcondrecov}.

\subsection{The one-dimensional case and a thin-film dimension reduction limit}

For 
$\mathbb S^1$-valued maps $u=(u_1,u_2)$ depending only on one variable, say $x_1\in I$ for a finite interval $I\subset\R$, 
we have
$\dv u =\frac{du_1}{dx_1}.$
Since $|u_2|=\sqrt{1-u_1^2}$,
and $\sqrt\cdot$ is $C^{1/2}$,
it is not so surprising that $L^2$ control on $\dv u$ 
implies fractional regularity of order $1/2$ for $u$,
as in Theorem~\ref{t:reg_estim}.
In fact, using this simple idea, 
one can completely determine the
$\Gamma$-limit 
  \cite[Definition~1.5]{braides}
 in this one-dimensional case.

\begin{theorem}\label{t:1D}
The energy functionals
$E_\e\colon \mathcal D'(I;\R^2)\to \R$ given for $\e>0$ by
\begin{align*}
%\label{eq:Eeps1D}
E_\e^{\mathrm{1D}}(u)
&
=
\begin{cases}
\int_I \big[ (u_1')^2 +\e (u_2')^2\big]\, dx
&
\text{ if }
u\in H^1(I;\mathbb S^1)\,,
\\
+\infty & \text{ otherwise}\,,
\end{cases}
\end{align*}
$\Gamma$-converge, as $\e\to 0$,
 to the limit functional $E_0\colon\mathcal D'(I;\R^2)\to\R$ given by
\begin{align*}
E_0^{\mathrm{1D}}(u)
&
=
\begin{cases}
\int_I (u_1')^2\, dx
&
\text{ if }u\in B^{1/2}_{4,\infty}(I;\mathbb S^1)
\text{ and }u_1\in H^1(I)\,,
\\
+\infty
&
\text{ otherwise}\,.
\end{cases}
\end{align*}
Moreover, any  sequence $(u_\e)\subset H^1(I;\mathbb S^1)$ with $\sup_\e E^{\mathrm{1D}}_\e(u_\e) <\infty$ for some sequence $\e\to 0$ is bounded in 
$B^{1/2}_{4,\infty}(I;\mathbb S^1)$ and therefore compact in $W^{s,4}(I;\mathbb S^1)$ for any $s\in (0,1/2)$.
\end{theorem}

The main subtlety in the proof of Theorem~\ref{t:1D} lies in the construction of the recovery sequence.
In contrast with the two-dimensional case, where we had to impose \eqref{eq:suffcondrecov},
here we combine the 
$B^{1/2}_{4,\infty}-$regularity with the assumption that $\dv u=du_1/dx_1\in L^2$ 
in order to deduce the validity of \eqref{eq:suffcondrecov}.

We would like to use this one-dimensional observation 
in order to gain more insight into the
 two-dimensional case.
Therefore we consider a somewhat simplified model
which resembles the one-dimensional model,
while introducing some mild two-dimensionality.
 
Specifically, for $\e,\delta>0$, we study 
the energy $E_\e$ given by \eqref{eq:Eeps}
 in a thin film
$\Omega_\delta =(-1,1)\times (-\delta,\delta),$
subject to Dirichlet boundary conditions in the first variable and periodic conditions in the thin variable $x_2$.
After rescaling the $x_2$ variable to unit length,
we are led to the energy functional
\begin{align*}
%\label{eq:Eepsdelta}
E_{\e,\delta}(u)
&
=\delta E_\e(u(\cdot,\cdot/\delta);\Omega_\delta)
\nonumber
\\
&
=\frac 12\int_{[-1,1]^2} 
\left(\partial_1u_{1} + \frac 1\delta \partial_2 u_2\right)^2 
+
\e \left( \partial_1 u_2 -\frac 1\delta \partial_2 u_1\right)^2\, dx\,,
\end{align*}
for maps $u\in H^1( (-1,1)^2;\mathbb S^1)$ with boundary conditions
\begin{align}
u(\pm 1,x_2)
&=e_\pm \in\mathbb S^1\qquad\forall x_2\in (-1,1)\,,
\nonumber
\\
u(x_1,+1)
&
=u(x_1,-1)\qquad\forall x_1\in (-1,1).
\label{eq:thin_bdry}
\end{align}
The latter imposes periodicity with respect to the second variable.
We can consider the energy $E_{\e,\delta}\colon \mathcal D'((-1,1)^2;\R^2)\to \R$ by setting it to be $+\infty$ unless 
$u\in H^1( (-1,1)^2;\mathbb S^1)$ with the boundary conditions \eqref{eq:thin_bdry}.

%The energy density in \eqref{eq:Eepsdelta} corresponds, in rescaled variables, 
%to the energy density $(\nabla\cdot u)^2+\e (\nabla\times u)^2$ in a thin slab $(-1,1)\times (-\delta,\delta)$.

\begin{theorem}\label{t:thin}
The $\Gamma$-limit in $\mathcal D'((-1,1)^2;\R^2)$ as $(\e,\delta)\to (0,0)$ of the energy functionals $E_{\e,\delta}$
is  the functional 
$\widetilde E_0^{\mathrm{1D}}\colon \mathcal D'((-1,1)^2;\R^2)\to\R$ given by
 \begin{align*}
 \widetilde E_0^{\mathrm{1D}}(u)
 =
 \begin{cases}
\int_{-1}^{1} (\partial_1 u_1)^2\, dx_1
&
\text{ if }
\partial_{2}u=0,\;
u\in B^{\frac 12}_{4,\infty}((-1,1);\mathbb S^1),
\\
&\quad\text{ and }u(\pm 1 )= e_\pm\,,
\\
+\infty
&
\text{ otherwise}.
\end{cases}
\end{align*}
Moreover any sequence $(u^{(j)})\subset H^1((-1,1)^2;\mathbb S^1)$ such that $\sup E_{\e_j,\delta_j}(u_j)<\infty$ for some $(\e_j,\delta_j)\to (0,0)$ is bounded in $B^{1/2}_{3,\infty}((-1,1)^2;\mathbb S^1)$, hence compact in 
$W^{s,3}((-1,1)^2;\mathbb S^1)$ for any $s\in (0,1/2)$.
\end{theorem}

Most of Theorem~\ref{t:thin} follows directly from Theorem~\ref{t:reg_estim} and Theorem~\ref{t:1D}, 
except for one slightly subtle part: the fact that bounded energy limits  have the $B^{1/2}_{4,\infty}-$regularity, see Lemma~\ref{l:1Dcont}.

\begin{remark}\label{r:sym}
Theorem~\ref{t:thin} implies, in particular, that minimizers of $E_{\e,\delta}$ converge to a one-dimensional configuration.
But one can actually also check directly that minimizers of $E_{\e,\delta}$ are one-dimensional for any value of $\e,\delta>0$, see  Appendix~\ref{a:sym}.
\end{remark}

\subsection{Plan of the article}

In \textsection\ref{s:comp_reg} we prove the compensated regularity estimate of Theorem~\ref{t:reg_estim}, 
and Corollary~\ref{c:jac} about the Jacobian.
In \textsection\ref{s:trace} we prove Theorem~\ref{t:trace} 
about the compactness of traces.
In \textsection\ref{s:recov} we construct the recovery sequence of Theorem~\ref{t:recovery2D}.
In \textsection\ref{s:1D} we treat the one-dimensional case and prove Theorem~\ref{t:1D}.
In \textsection\ref{s:thin} we prove the Theorem~\ref{t:thin} about the thin-domain limit.
In Appendix~\ref{a:traces} we adapt to our context the proof, due to \cite{vasseur01}, of the fact that limit maps admit strong traces.
In Appendix~\ref{a:sym} we prove the one-dimensional symmetry of minimizers with periodic boundary conditions.
Throughout the article, we will use the symbol $\lesssim$ to denote inequality up to an absolute multiplicative constant.

\subsection*{Acknowledgments}

Part of this work was conducted during an AIM-SQuaREs meeting and AIM support is gratefully acknowledged.
X.L. is supported by the ANR project ANR-22-CE40-0006. L.B. is supported by an NSERC Discovery grant. P.S. is supported by a Simons Collaboration grant 585520 and an NSF grant DMS 2106516. D.G. is supported by an NSF grant DMS-2106551.

\section{Compensated regularity}\label{s:comp_reg}

In this section we prove  Theorem~\ref{t:reg_estim}.
The proof relies on the strategy 
used in \cite{GL20} 
to show regularity of a class of 
divergence-free unit vector fields,
based on ideas from \cite{golse09,GP13}
where the authors introduced quantitative improvements on 
the div-curl lemma from the theory of compensated compactness  \cite{murat78,tartar79}.
Before going into the details, we describe briefly 
the structure and main ingredients of the proof:
\begin{itemize}
\item The assumptions on $u$ imply $\dv\Phi(u)\in L^2(\Omega)$ for a large family of vector fields $\Phi\in C^1(\mathbb S^1;\R^2)$, called \emph{entropies} by analogy with scalar conservations laws,
see \eqref{eq:chainENT} and Lemma~\ref{l:kinlim}.
\item These $L^2$ entropy productions provide a 
\emph{kinetic equation} satisfied by the characteristic function $\chi(s,x)=\mathbf 1_{e^{is}\cdot u(x) >0}$, where $s\in\T=\R/2\pi\Z$ is the kinetic variable, see Lemma~\ref{l:kin}.
\item This kinetic equation implies a \emph{compensation identity} relating the divergnce $\dv u\in L^2$ to a family of quantities $\Delta^{\varphi,\chi}(x,h)$ depending quadratically on the finite difference $D^h_x\chi$,
see \eqref{eq:Delta} and Lemma~\ref{l:comp_id}.
\item For a particular choice of the parameter $\varphi$, 
the quantity $\Delta^{\varphi,\chi}$ controls 
variations of $u$ via a 
coercivity estimate  
$\Delta^{\varphi,\chi}(x,h)\gtrsim |D^h u(x)|^3$, see  \eqref{eq:Delta_varphi0_coerc}.
\item Estimating the terms in the compensation identity provides control on the integral $\int_{B_1} \Delta^{\varphi,\chi}(x,h)\,dx$ in terms of $\|\dv u\|_{L^2(B_2)}$, which, combined with the previous ingredient, proves Theorem~\ref{t:reg_estim}.
\end{itemize}
In the next subsections we detail all these ingredients: entropies and kinetic formulation are presented in \textsection\ref{ss:ent_kin}, 
the compensation identity in \textsection\ref{ss:comp_id},
estimates of the terms involved in that identity in \textsection\ref{ss:RHScomp},
to conclude the proof of the regularity estimate \eqref{eq:reg_estim} in \textsection\ref{ss:pf_reg_kin}.

\subsection{Entropies and kinetic formulation}\label{ss:ent_kin}

We rely on the notion of entropies for the eikonal equation, introduced in \cite{DKMO01} in analogy with scalar conservation laws.
These are the vector fields in the class
\begin{align}
\mathrm{ENT}=\bigg\lbrace
\Phi\in C^1(\mathbb S^1;\R^2)\colon
&
\exists\lambda_\Phi\in C^0(\mathbb S^1;\R),
\nonumber
\\
&
\frac{d}{d\theta}\Phi(e^{i\theta})=\lambda_\Phi(e^{i\theta})ie^{i\theta}
\quad\forall\theta\in\R
\bigg\rbrace\,.
\label{eq:ENT}
\end{align}
Their relevance comes from a direct application of the chain rule, which shows that, for all $u\in H^1(\Omega;\mathbb S^1)$, we have
\begin{align}\label{eq:chainENT}
\dv \Phi(u)=\lambda_\Phi(u)\dv u\qquad \forall\Phi\in \mathrm{ENT}\,.
\end{align}
%and from the fact that there are ``many'' entropies.
%
In the context of hyperbolic conservation laws, it
was first observed in \cite{LPT94} that 
information about a large family of entropy productions, as in \eqref{eq:chainENT}, can be succinctly expressed as a kinetic formulation, see also \cite{perthame02}.
In the context of divergence-free unit vector fields, 
a kinetic formulation 
was introduced in \cite{JP01} and reformulated in \cite{GL20},
for the
indicator function
\begin{align}\label{eq:chi}
\chi(s,x)=\mathbf 1_{u(x)\cdot e^{is}>0}\,,\quad s\in\T=\R/2\pi\Z,\; x\in\Omega\,,
\end{align} 
where
$z_1\cdot z_2 =\re(\bar z_1 z_2)$ is the scalar product of $z_1,z_2\in\mathbb C$.
We adapt it here to our setting.

\begin{lemma}\label{l:kin}
For any $u\in H_{\dv}(\Omega;\mathbb S^1)$ satisfying \eqref{eq:chainENT},
the indicator function $\chi\in L^\infty(\T\times\Omega)$ defined in \eqref{eq:chi}
satisfies
\begin{align}
&
e^{is}\cdot\nabla_x\chi=\Theta
\quad\text{in }\mathcal D'(\T\times\Omega)\,,
\label{eq:kin}
\end{align}
for $\Theta\in L^2(\Omega;\mathcal M(\T))$
 given explicitly by
\begin{align*}
&\Theta(s,x)=
\left(\delta_{\theta(x)+\frac\pi 2}(s) + \delta_{\theta(x)-\frac\pi 2}(s)  \right)\dv u (x) \,,
\nonumber
\end{align*}
where $\theta(x)\in\T$ is such that $u(x)=e^{i\theta(x)}$.
\end{lemma}

\begin{proof}[Proof of Lemma~\ref{l:kin}]
Consider 
 the family of entropies given by
\begin{align}\label{eq:Phi_g}
\Phi_g(e^{i\theta})
&
=\int_{\T} g(s)\mathbf 1_{e^{is}\cdot e^{i\theta}>0}\, e^{is}\, ds
\nonumber
\\
&
=
\int_{\theta-\frac\pi 2}^{\theta+\frac\pi 2}
g(s) e^{is}\, ds
\,, \qquad g\in C^0(\T;\R)\,,
%\nonumber
\end{align}
which satisfies \eqref{eq:ENT} with $\lambda_{\Phi_g}(e^{i\theta})=g(\theta+\pi/2)+g(\theta-\pi/2)$.
%Here, $z_1\cdot z_2 =\re(\bar z_1 z_2)$ is the scalar product of $z_1,z_2\in\mathbb C$,
% $\T$ denotes the flat one-dimensional torus $\T=\R/2\pi\Z$,
%and we will often identify a variable $e^{i\theta}\in\mathbb S^1$ with $\theta\in\T$.
%
The family \eqref{eq:Phi_g} is naturally obtained by taking a $2\pi$-periodic antiderivative of the identity $\partial_\theta\Phi =\lambda_\Phi ie^{i\theta}$
in the case where $\lambda_\Phi$ is $\pi$-periodic,
and it has the advantage of being defined in terms of the singular entropies
\begin{align*}
\Phi^{\lbrace s\rbrace}(z)=\mathbf 1_{z\cdot e^{is}>0}e^{is}\,,
\end{align*}
which play an important role in \cite{DKMO01,JOP02}.
This enables us to obtain the kinetic formulation \eqref{eq:kin}.
Namely, testing \eqref{eq:chainENT} against a test function $\zeta(x)$ for all entropies in the family \eqref{eq:Phi_g}, we have
\begin{align*}
\int_{\Omega} \langle \Phi_g(u(x)),\nabla\zeta(x)\rangle\, dx
&
=
\int_{\Omega}\lambda_{\Phi_g}(u(x))\dv u(x)\, \zeta(x)\, dx\,.
\end{align*}
Using \eqref{eq:Phi_g} and the expression of 
\begin{align*}
\lambda_{\Phi_g}(e^{i\theta})=
g(\theta+\frac\pi 2)+g(\theta-\frac\pi 2)
= \int_{\T} g(s) (\delta_{\theta+\frac\pi 2} +\delta_{\theta-\frac\pi 2})(ds)\,,
\end{align*}
this can be rewritten as
\begin{align*}
\int_{\T\times\Omega}\mathbf 1_{u(x)\cdot e^{is}>0}
\,e^{is}\cdot\nabla_x[ g(s)\zeta(x)]\, dsdx
=\langle \Theta,g(s)\zeta(x) \rangle
\,,
\end{align*}
for all $g\in C^0(\T)$ and $\zeta\in C^1_{c}(\Omega)$,
where $\langle \Theta,g(s)\zeta(x) \rangle$ denotes the value of the distribution $\Theta\in L^2(\Omega;\mathcal M(\T))\subset \mathcal D'(\T\times\Omega)$ 
applied to the function $(s,x)\mapsto g(s)\zeta(x)$.
%,
%where the distribution
%$\Theta\in L^2(\Omega;\mathcal M(\T))$
% is given by
%\begin{align*}
%&\Theta(s,x)=
%\left(\delta_{iu(x)}(e^{is}) + \delta_{-iu(x)}(e^{is})  \right)\dv u (x) \,.
%\nonumber
%\end{align*}
In other words,
the indicator function $\chi(s,x)=\mathbf 1_{u(x)\cdot e^{is}>0}$ satisfies
the kinetic formulation \eqref{eq:kin}.
\end{proof}

\begin{remark}\label{r:kin}
For another interpretation of the kinetic formulation \eqref{eq:kin},
note that
\begin{align*}
e^{is}\mathbf 1_{e^{is}\cdot z>0} =\partial_s\Psi(s,z) +\frac 1\pi z,
\end{align*}
where $\Psi(s,\cdot)$ is the  Lipschitz entropy given by
\begin{align*}
\Psi(s,z)=\int_{\T} f_0(t-s)\mathbf 1_{e^{it}\cdot z>0}\,e^{it}\, dt,
\qquad f_0(s)=\frac{s}{2\pi}\text{ for }0\leq s<2\pi\,.
\end{align*}
Therefore we have
\begin{align*}
e^{is}\cdot\nabla_x \chi
&=\dv [e^{is}\mathbf 1_{e^{is}\cdot u>0}] 
=\partial_s \dv \Psi(s,u) 
+ \frac 1\pi \dv u (x)\\
&=\left(\frac 1\pi + \partial_s [\lambda_{0}(e^{-is}u)]  \right)\dv u,
\end{align*}
where $\lambda_0(e^{i\theta})=f_0(\theta+\pi/2)+f_0(\theta-\pi/2)$.
Since $f_0'(s)=1/(2\pi)-\delta_0(s)$, this is exactly \eqref{eq:kin}.
\end{remark}

The kinetic formulation \eqref{eq:kin} 
is all we require 
to prove the compensated regularity 
estimate \eqref{eq:reg_estim}.
In other words,
we will obtain Theorem~\ref{t:reg_estim} as a consequence of 
the following slightly more general statement.

\begin{theorem}\label{t:reg_kin}
Let $u\in H_{\dv}(\Omega;\mathbb S^1)$
be such that 
%the indicator 
%function
$\chi(s,x)=\mathbf 1_{u(x)\cdot e^{is}>0}$
  satisfies the kinetic formulation \eqref{eq:kin}. 
%and with entropy productions satisfying \eqref{eq:chainENT}.
Then $u$ satisfies the
$B^{1/2}_{3,\infty,\loc}$
 regularity estimate \eqref{eq:reg_estim}.
\end{theorem}

\begin{proof}[Proof of Theorem~\ref{t:reg_estim} from Theorem~\ref{t:reg_kin}]
We have already seen that maps $u\in H^1(\Omega;\mathbb S^1)$ satisfy \eqref{eq:chainENT} and therefore \eqref{eq:kin}, so Theorem~\ref{t:reg_kin} implies Theorem~\ref{t:reg_estim} in the case $u\in H^1(\Omega;\mathbb S^1)$.
In order to treat the case $u\in\mathcal S$ we just need to check that Theorem~\ref{t:reg_kin} also implies that limits of bounded energy sequences satisfy \eqref{eq:chainENT} and therefore \eqref{eq:kin}. 
By definition \eqref{eq:Glimsupfinite} of $\mathcal S$, there exists a sequence $(u_j)\subset H^1(\Omega;\mathbb S^1)$ 
such that 
$u_j\to u$ in $\mathcal D'(\Omega)$ and 
$E_{\e_j}(u_j)\leq C$ for some $\e_j\to 0$.
This energy bound implies $\|\dv u_j\|_{L^2(\Omega)}^2\leq C$,
so, according to the next Lemma, $u$ does indeed satisfy \eqref{eq:chainENT}.
\end{proof}

\begin{lemma}\label{l:kinlim}
If a sequence $(u_j)\subset H^1(\Omega;\mathbb S^1)$ satisfies $\sup_j\|\dv u_j\|_{L^2(\Omega)}<\infty$, there exists $u\in B^{1/2}_{3,\infty,\loc}(\Omega;\mathbb S^1)$ with $\dv u\in L^2(\Omega)$ such that
\begin{align*}
&
u_j\to u\qquad\text{in }L^1(\Omega)\text{ and }a.e.,
\\
&
\dv \Phi(u_j)\rightharpoonup \dv\Phi(u)=\lambda_\Phi(u)\dv u\quad
\text{weakly in }L^q(\Omega)
\text{ for }1 < q<2\,,
\end{align*}
for all  $\Phi\in\mathrm{ENT}$,
along  a non-relabeled subsequence.
%Moreover, $u$ satisfies the kinetic formulation \eqref{eq:kin}.
\end{lemma}
\begin{proof}[Proof of Lemma~\ref{l:kinlim} from Theorem~\ref{t:reg_kin}]
Since $u_j\in H^1(\Omega;\mathbb S^1)$, it satisfies \eqref{eq:chainENT} and therefore the kinetic formulation \eqref{eq:kin} by Lemma~\ref{l:kin}.
We may therefore apply Theorem~\ref{t:reg_kin},
ensuring that
the sequence $(u_j)$ is bounded in
$ B^{1/2}_{3,\infty,\loc}(\Omega;\mathbb S^1)$.
By compactness of the embedding $B^{1/2}_{3,\infty,\loc}\subset L^1_{\loc}$, we can extract a subsequence converging a.e. in $\Omega$, hence in $L^1(\Omega)$ by dominated convergence.
Invoking also the weak compactness of the sequence $(\dv u_j)\subset L^2(\Omega)$,
we can therefore find
$u\in B^{1/2}_{3,\infty,\loc}(\Omega;\mathbb S^1)$
such that $\dv u\in L^2(\Omega)$ 
and
\begin{align*}
&
u_j\to u\qquad\text{in }L^1(\Omega)\text{ and }a.e.,
\\
&
\dv u_j\rightharpoonup \dv u \qquad\text{weakly in }L^2(\Omega)\,,
\end{align*}
along a non-relabeled subsequence.
%Note that the convergence $u_j\to u$ a.e. actually implies convergence in $L^1(\Omega)$ by dominated convergence.
Applying \eqref{eq:chainENT} to $u_j$, we have
\begin{align*}
\dv\Phi(u_j)=\lambda_\Phi(u_j)\dv u_j \qquad\forall \Phi\in\mathrm{ENT}\,.
\end{align*}
Since $\dv u_j\rightharpoonup \dv u$ weakly in $L^2(\Omega)$ and $\lambda_\Phi(u_j)\to \lambda_\Phi(u)$ strongly in $L^p(\Omega)$ for any $p\in [1,\infty)$, we
have
\begin{align*}
\lambda_\Phi(u_j)\dv u_j\rightharpoonup \lambda_\Phi(u)\dv u\quad\text{weakly in }L^q(\Omega)
\text{ for }1 < q<2\,.
\end{align*}
We also have
$\dv\Phi(u_j)\to\dv \Phi(u)$ in $\mathcal D'(\Omega)$ 
since $\Phi(u_j)\to\Phi(u)$ in $L^1(\Omega)$,
hence 
\begin{align*}
\dv \Phi(u_j)\rightharpoonup \dv\Phi(u)=\lambda_\Phi(u)\dv u\quad
\text{weakly in }L^q(\Omega)
\text{ for }1 < q<2\,,
\end{align*}
for all  $\Phi\in\mathrm{ENT}$.
%Then, arguing exactly as in Lemma~\ref{l:kin}, we see that $u$ satisfies the kinetic formulation \eqref{eq:kin}.
\end{proof}

In the next subsections 
we provide the proof of 
Theorem~\ref{t:reg_kin},
starting with the compensation identity relating 
finite differences $D^h_x\chi$ with $\dv u$.

\subsection{A compensation identity}\label{ss:comp_id}

We denote by $T^h$ and  $D^h$ the translation and finite difference operators in the $x$ variable:
\begin{align}\label{eq:ThDh}
T^h\! f(x)=f(x+h),\quad D^h\! f =T^h\! f-f\,,
\end{align}
for any function $f(x)$,
and we will frequently use the identity
\begin{align}\label{eq:Dh_prod}
D^h(fg)=f D^h g +T^h g D^h f\,.
\end{align}
As noticed first in \cite{GP13} in a slightly different context, 
the kinetic formulation \eqref{eq:kin}
provides an identity 
for some  bilinear functions of $D^h\chi$,
see \eqref{eq:Delta}. 
This identity can be thought of as a 
generalization
of the integration by parts at the origin of compensated compactness phenomena \cite{murat78,tartar79},
see Remark~\ref{r:comp_id}.
Note also that this compensation identity does not use the specific form of the characteristic function $\chi$ in Lemma~\ref{l:kin},
but is valid for any  function $\chi$ satisfying the kinetic formulation \eqref{eq:kin}.

For any
 $\varphi\in L^1(\T)$ odd and $\pi$-periodic,
 and for the indicator function $\chi(s,x)=\mathbf 1_{u(x)\cdot e^{is}>0}$,  we define the function
\begin{align}\label{eq:Delta}
\Delta^{\varphi,\chi}(x,h)
=\frac 12 \int_{\T^2} \varphi(t-s)D^h\chi(t,x) D^h\chi(s,x)\sin(t-s)\, dtds,
\end{align}
 on $\Omega^*=\lbrace (x,h)\in \Omega\times\R^2\colon x+h\in\Omega\rbrace$.
(Note that this quantity is zero if $\varphi\in L^1(\T)$ is even or anti-$\pi$-periodic.)
The same quantity (without the 1/2 factor) is considered in \cite[\S~3.2]{GL20}, 
for the specific choice of $\varphi$ given by
\begin{align*}%\label{eq:varphi0}
\varphi_0(s)=\mathbf 1_{\cos(s)\sin(s)>0}-\mathbf 1_{\cos(s)\sin(s)<0}\,,
\end{align*}
  as can be seen by taking $(\xi,\eta)=(e^{is},e^{it})$ in the expression of $\Delta$ in  \cite[\S~3.2]{GL20}.
It is shown in \cite[Lemma~3.8]{GL20} that we have
\begin{align}\label{eq:Delta_varphi0_coerc}
\Delta^{\varphi_0,\chi}(x,h) \gtrsim |D^hu|^3(x)\,.
\end{align}
Given this coercivity property, it is clear that the regularity estimate \eqref{eq:reg_estim} will follow 
from an estimate on the integral of $\Delta^{\varphi_0,\chi}$ 
in terms of $\dv u$.
This is made possible by Lemma~\ref{l:kin} below, 
where we establish a compensation identity relating $\Delta^{\varphi,\chi}$ to $\dv u$ via the distribution $\Theta$ in the kinetic formulation \eqref{eq:kin}.

\begin{remark}\label{r:comp_id}
It was realized in \cite{GP13} 
(and adapted to $\mathbb S^1$-valued maps in \cite{GL20})
 that the kinetic formulation \eqref{eq:kin} 
provides a compensation identity satisfied by the
 quantity $\Delta^{\varphi,\chi}$, 
and which allows one to estimate its integral
 in an efficient way.
In the context of Burgers' equation, 
this compensation identity has been interpreted in \cite{GJO15} as a modification of the so-called Kármán-Howarth-Monin formula in fluid dynamics.
Another way to understand why some compensations can be expected in the quantity $\Delta^{\varphi,\chi}$ is to rewrite it as
\begin{align*}
\Delta^{\varphi,\chi} =\int_{\T^2}\varphi(t-s)D^h[\Phi^{\lbrace t\rbrace}(u)]\wedge D^h[\Phi^{\lbrace s\rbrace}(u)]\, dtds\,,
\end{align*}
where $\Phi^{\lbrace s\rbrace}(z)=e^{is}\mathbf 1_{z\cdot e^{is}>0}$ is the singular entropy which already appeared in the proof of Lemma~\ref{l:kin},
and $X\wedge Y=\det(X,Y)$ for any $X,Y\in\R^2$.
Classical div-curl estimates (see. e.g. \cite[\S~3.4]{golse09})
provide improved  control
 on the integral of $A\wedge B$ for vector fields $A,B$ with controlled divergence,
and would therefore provide control on the integral of
 $\Delta^{\varphi,\chi}$ if the entropies $\Phi^{\lbrace s\rbrace}$ were regular, 
as in \cite[\S~4.3]{GL20} or \cite[Proposition~6.1]{GMPS24}.
In that sense, the compensation identity \eqref{eq:comp_id} satisfied by $\Delta^{\varphi,\chi}$
 can also be interpreted as a formula of div-curl type, adapted to these singular entropies.
 \end{remark}

\begin{lemma}\label{l:comp_id}
Let $\Omega\subset\R^2$, $\T=\R/2\pi\Z$,
$\chi\in L^\infty(\T\times \Omega)$
such that $x\mapsto \chi(s,x)$ is $C^1$, and $\Theta\in L^1(\T\times\Omega)$ such that
\begin{align}\label{eq:g_kin}
e^{is}\cdot\nabla_x \chi =\Theta
\quad\text{a.e. in }\T\times\Omega\,.
\end{align}
Then, for any
 $\varphi\in L^1(\T)$ odd and $\pi$-periodic, the function
 $\Delta^{\varphi,\chi}$ defined in \eqref{eq:Delta}
%\begin{align*}
%\Delta^{\varphi,\chi}(x,h)
%=\frac 12 \int_{\T^2} \varphi(t-s)D^h\chi(t,x) D^h\chi(s,x)\sin(t-s)\, dtds,
%\end{align*}
%defined on $\Omega^*=\lbrace (x,h)\in \Omega\times\R^2\colon x+h\in\Omega\rbrace$,
satisfies
\begin{align}\label{eq:comp_id}
\frac{d}{d\tau}\Delta^{\varphi,\chi}(x,\tau\mathbf e_1)
=I^\tau(x) +\dv_x A^\tau (x)\,,
\end{align}
where $I^\tau$ and  $A^\tau=(A_1^\tau,A_2^\tau)$ are given by
\begin{align}
I^\tau(x)
&=
\int_{\T^2} (\Theta+T^{\tau\mathbf e_1} \Theta)(s,x) \varphi(t-s)D^{\tau\mathbf e_1}\chi(t,x)\sin(t) \, ds dt
\nonumber
\\
&\quad
-D^{\tau \mathbf e_1}\bigg[ \int_{\T^2} \Theta(s,x) \varphi(t-s)\chi(t,x)\sin (t) \, ds dt
\bigg]
\label{Itau}
\\
A_1^\tau(x)
&=
\int_{\T^2} \varphi(t-s)\sin(t)  \cos(s) 
T^{\tau\mathbf e_1}\chi(t,x)D^{\tau\mathbf e_1}\chi(s,x) \, ds dt
\label{Aone}
\\
A_2^\tau(x)
&
=\int_{\T^2} \varphi(t-s)\sin(t)\, \sin(s)\,
 \chi(t,x)D^{\tau\mathbf e_1}\chi(s,x)\, dsdt\,.\label{Atwo}
\end{align}
Moreover, 
the identity \eqref{eq:comp_id} is 
valid in the sense of distributions if 
 $\chi\in L^\infty(\T\times\Omega)$ 
 and $\Theta\in L^2(\Omega;\mathcal M(\T))$
satisfy \eqref{eq:g_kin}, 
and $\varphi\in C^0(\T)$ is odd and $\pi$-periodic. 
\end{lemma}

The compensation identity in Lemma~\ref{l:comp_id} is the same as in \cite[Lemma~3.8]{GL20}, 
but the quantity $I^\tau$ in \eqref{eq:comp_id} 
is rewritten
in a way that makes further compensations appear.
These further compensations were not needed in \cite{GL20}, but were already used in \cite[Lemma~4.8]{LP23}.

\begin{proof}[Proof of Lemma~\ref{l:comp_id}]
The last statement follows from applying the first statement to $\chi$ and $\Theta$ mollified with respect to $x$, which preserves \eqref{eq:g_kin}, and passing to the limit in \eqref{eq:comp_id}.
Thus we assume that $x\mapsto \chi(s,x)$ is $C^1$ and that $\Theta\in L^1(\T\times\Omega)$.

The identity \eqref{eq:comp_id} follows  from the calculations in \cite[Lemma~3.9]{GL20},
which we reproduce here for the readers' convenience.
To simplify notations we simply write $\Delta=\Delta^{\varphi,\chi}(x,\tau \mathbf e_1)$.
First we  have
\begin{align*}
\frac{d}{d\tau}\Delta
&
=
\frac 12 \int_{\T^2} \varphi(t-s)T^{\tau\mathbf e_1}\partial_1\chi(t,x) D^{\tau\mathbf e_1}\chi(s,x)\sin(t-s)\, dtds
\\
&
\quad
+
\frac 12 \int_{\T^2} \varphi(t-s)D^{\tau\mathbf e_1}\chi(t,x) T^{\tau\mathbf e_1}\partial_1\chi(s,x)\sin(t-s)\, dtds\,.
\end{align*}
Changing variable $(s,t)\mapsto (t,s)$ in the second integral and using the fact that $\varphi$ is odd, this becomes
\begin{align*}
\frac{d}{d\tau}\Delta
&
=
 \int_{\T^2} \varphi(t-s)T^{\tau\mathbf e_1}\partial_1\chi(t,x) D^{\tau\mathbf e_1}\chi(s,x)\sin(t-s)\, dtds
 \\
 &
=
 \int_{\T^2} \varphi(t-s)T^{\tau\mathbf e_1}\partial_1\chi(t,x) D^{\tau\mathbf e_1}\chi(s,x)\sin(t)\cos(s)\, dtds
 \\
 &\quad
 -
  \int_{\T^2} \varphi(t-s)T^{\tau\mathbf e_1}\partial_1\chi(t,x) D^{\tau\mathbf e_1}\chi(s,x)\cos(t)\sin(s)\, dtds
  \\
  &
=\partial_1 A_1^\tau
-\int_{\T^2} \varphi(t-s)T^{\tau\mathbf e_1}\chi(t,x) D^{\tau\mathbf e_1}\partial_1\chi(s,x)\sin(t)\cos(s)\, dtds
\\
&
\quad
 -
  \int_{\T^2} \varphi(t-s)T^{\tau\mathbf e_1}\partial_1\chi(t,x) D^{\tau\mathbf e_1}\chi(s,x)\cos(t)\sin(s)\, dtds\,,
\end{align*}
where $A_1^\tau$ the first component of the vector field $A^\tau$ in \eqref{eq:comp_id}.
Using again the oddness of $\varphi$ and the change of variable $(s,t)\mapsto (t,s)$ in the last integral, 
we find
\begin{align*}
\frac{d}{d\tau}\Delta
&
=\partial_1 A_1^\tau
+\int_{\T^2} \varphi(t-s)T^{\tau\mathbf e_1}\chi(t,x)\partial_1\chi(s,x)\sin(t)\cos(s)\, dtds
\\
&
\quad
-  \int_{\T^2} \varphi(t-s)\chi(t,x)
T^{\tau\mathbf e_1}\partial_1\chi(s,x) \cos(s)\sin(t)\, dtds\,.
\end{align*}
Now we use the kinetic equation \eqref{eq:g_kin}, which gives
\begin{align*}
\cos(s)\partial_1\chi(s,x) =-\sin(s)\partial_2\chi(s,x) +\Theta(s,x)\,,
\end{align*}
to rewrite the above as
\begin{align*}
\frac{d}{d\tau}\Delta
&
=\partial_1 A_1^\tau
+\int_{\T^2}\Theta(s,x) \varphi(t-s)T^{\tau\mathbf e_1}\chi(t,x)\sin(t)\, dtds
\\
&
\quad
-  \int_{\T^2}T^{\tau \mathbf e_1}\Theta(s,x) \varphi(t-s)\chi(t,x)
\sin(t)\, dtds
\\
&
\quad
-\int_{\T^2} \varphi(t-s)T^{\tau\mathbf e_1}\chi(t,x)\partial_2\chi(s,x)\sin(t)\sin(s)\, dtds
\\
&
\quad
+  \int_{\T^2} \varphi(t-s)\chi(t,x)
T^{\tau\mathbf e_1}\partial_2\chi(s,x) \sin(s)\sin(t)\, dtds\,.
\end{align*}
Using again the oddness of $\varphi$ and 
the change of variable $(s,t)\mapsto (t,s)$, we see that the last two lines correspond to $\partial_2 A_2^\tau$,
where $A_2^\tau$ is the second component of the vector field $A^\tau$ in \eqref{eq:comp_id}.
Rewriting slightly the first two integrals, we are left with
\begin{align*}
\frac{d}{d\tau}\Delta
&
=\dv  A^\tau
+\int_{\T^2}\Theta(s,x) \varphi(t-s)D^{\tau\mathbf e_1}\chi(t,x)\sin(t)\, dtds
\\
&
\quad
-  \int_{\T^2}D^{\tau \mathbf e_1}\Theta(s,x) \varphi(t-s)\chi(t,x)
\sin(t)\, dtds\,.
\end{align*}
Using the identity \eqref{eq:Dh_prod} to rewrite the last line, we obtain exactly \eqref{eq:comp_id}.
\end{proof}

%\begin{remark}\label{r:comp_id}
%Note that convolution with respect to $x$ preserves the equation \eqref{eq:g_kin},
%so applying Lemma~\ref{l:comp_id} to $\chi$ and $g$ mollified with respect to $x$,
% we deduce that 
%the identity \eqref{eq:comp_id} is 
%valid in the sense of distribution if 
%\begin{itemize}
%\item $\chi\in L^\infty(\T\times\Omega)$ 
%satisfies \eqref{eq:g_kin},
%\item $\Theta\in \mathcal D'(\T\times\Omega)$ is  a distribution of order $-1$,
%\item $\varphi\in C^1(\R)$ is odd, $\pi$-periodic, and such that
%$\int_\T \Theta(s,x)\varphi(t-s)\, ds$ (interpreted in the sense of distribution with respect to the variable $s$) is, as a distribution of the $(t,x)$ variable, a measure.
%\end{itemize}
%The last condition is satisfied in particular if $\Theta\in L^2(\Omega;\mathcal M(\T))$ as in \eqref{eq:kin}, or if $\Theta=\partial_s\sigma$ for some measure $\sigma\in\mathcal M(\T\times\Omega)$, as in \cite{GL20}.
%\end{remark}
%

\subsection{Estimating the right-hand side of the compensation identity}\label{ss:RHScomp}

For any odd $\pi$-periodic $\varphi\in C^0(\R)$
we can apply the compensation identity of Lemma~\ref{l:comp_id} 
to the kinetic formulation \eqref{eq:kin} obtained in Lemma~\ref{l:kin}.
In this subsection we estimate the quantities 
 $I^\tau$ and $A^\tau$ appearing in \eqref{eq:comp_id}.

\begin{lemma}\label{l:estim_A}
Let $u\colon\Omega\to\mathbb S^1$ measurable and $\chi(s,x)=\mathbf 1_{e^{is}\cdot u(x) >0}$. 
Then the vector field $A^\tau$ in \eqref{eq:comp_id} satisfies
\begin{align*}
|A^\tau|\lesssim \|\varphi\|_{L^1(\T)} \, |D^{\tau e_1} u|
\qquad\text{a.e. in }\Omega\cap(\Omega-\tau\mathbf e_1)\,.
\end{align*}
\end{lemma}
\begin{proof}[Proof of Lemma~\ref{l:estim_A}]
Recalling the definitions of $A_1^\tau$ and $A_2^\tau$ from \eqref{Aone} and \eqref{Atwo}, as well as the finite difference operator $D^h$ and the translation operator $T^h$ defined in \eqref{eq:ThDh}, one finds that $A_1^\tau$ and $A_2^\tau$
can be rewritten as
\begin{align*}
A_1^\tau
&
=F_1(T^{\tau\mathbf e_1}u,T^{\tau\mathbf e_1}u)-F_1(T^{\tau\mathbf e_1}u,u)\,,
\\
A_2^\tau
&
=
F_2(u,T^{\tau\mathbf e_1}u)-F_2(u,u)\,,
\end{align*}
where $F_1,F_2\colon \mathbb S^1\times\mathbb S^1\to\R$ are given by
\begin{align*}
F_1(e^{i\theta},e^{i\alpha})
&
=\int_{\alpha-\frac\pi 2}^{\alpha+\frac\pi 2}
\int_{\theta-\frac\pi 2}^{\theta+\frac\pi 2}
\varphi(t-s) \sin (t) \, dt\, \cos (s) \, ds\,,
\\
F_2(e^{i\theta},e^{i\alpha})
&
=\int_{\alpha-\frac\pi 2}^{\alpha+\frac\pi 2}\int_{\theta-\frac\pi 2}^{\theta+\frac\pi 2}\varphi(t-s) \sin (t) \, dt\,\sin (s) \, ds\,.
\end{align*}
These satisfy
\begin{align*}
\left|\frac{d}{d\alpha}F_j(e^{i\theta},e^{i\alpha})
\right| \lesssim  \int_{\T}|\varphi|\, ds,
\end{align*}
and this implies the claimed inequality on $A^\tau$.
\end{proof}

\begin{lemma}\label{l:estim_I}
Let $u\in H_{\dv}(\Omega;\mathbb S^1)$ and $\chi(s,x)=\mathbf 1_{e^{is}\cdot u(x) >0}$. 
Then, for any $\eta\in C_c^1(\Omega;[0,1])$, 
the function $I^\tau$ in \eqref{eq:comp_id} satisfies
\begin{align*}
\int_{\Omega} I^\tau\,\eta^2\, dx
&
 \lesssim  
\|\varphi\|_{L^\infty(\T)}
\|\eta^2
 |D^{\tau\mathbf e_1}u|
 \big\|_{L^2(\Omega)} 
\|\dv u\|_{L^2(\Omega)}
\\
&
\quad
+ 
|\tau | 
\|\varphi\|_{L^1(\T)}
\|\nabla \eta\|_\infty 
\|\dv u\|_{L^1(\Omega)}
\,,
\end{align*}
for $|\tau|\leq \dist(\supp(\eta),\partial\Omega)$.
\end{lemma}
\begin{proof}
Integrating the expression \eqref{Itau} for $I^\tau$
 against $\eta^2$ and performing a discrete integration by parts using \eqref{eq:Dh_prod},
we obtain
\begin{align*}
\int_\Omega I^\tau\,\eta^2\, dx
&
=
\int_\Omega\int_{\T} (\Theta+T^{\tau\mathbf e_1} \Theta)(s,x) D^{\tau\mathbf e_1}\Lambda(s,x) \, ds \,\eta^2\, dx
\\
&
\quad
+
\int_\Omega\int_{\T} \Theta(s,x) T^{\tau\mathbf e_1}\Lambda(s,x) \, ds \, D^{\tau\mathbf e_1}[\eta^2]\, dx\,,
\\
\text{where }
\Lambda(s,x)&=\int_\T\varphi(t-s)\chi(t,x)\sin(t) \,  dt
\,.
\end{align*}
Recalling that  $\chi(t,x)=\mathbf 1_{e^{it}\cdot u(x)>0}$,
we see that
\begin{align*}
\Lambda(s,x)=F(e^{is},u(x)),\quad
F(e^{is},e^{i\theta})=\int_{\theta-\frac\pi 2}^{\theta +\frac\pi 2}\varphi(t-s)\sin(t)\, dt\,.
\end{align*}
(Note that $F$ is well-defined: this expression does not depend on the choice of $s$ modulo $2\pi$ since $\varphi$ is $2\pi$-periodic, nor on the choice of $\theta$ modulo $2\pi$ since 
$\int_0^{2\pi}\varphi(t-s)\sin(t)\,dt =0$  as a consequence of the $\pi$-periodicity of $\varphi$.)
From the inequalities
\begin{align*}
|F|\leq \|\varphi\|_{L^1(\T)},
\quad 
\bigg|\frac{d}{d\theta}F(e^{is},e^{i\theta})\bigg|
\leq 2\|\varphi\|_{L^\infty(\T)}\,,
\end{align*}
we deduce therefore 
\begin{align*}
\sup_{s\in\T}|\Lambda(s,x)| &\lesssim \|\varphi\|_{L^1(\T)}\,,
%\\
\quad
\sup_{s\in\T}|D^{\tau \mathbf e_1} \Lambda(s,x)|
%&
\lesssim \|\varphi\|_{L^\infty(\T)} |D^{\tau\mathbf e_1}u|(x)\,.
\end{align*}
Since the explicit expression 
\begin{align*}
&\Theta(s,x)=
\big(\delta_{\theta(x)+\frac\pi 2}(s)
 + \delta_{\theta(x)-\frac\pi 2}(s)  \big)\dv u (x) \,,
\end{align*}
obtained in Lemma~\ref{l:kin}
implies
\begin{align*}
\int_\T \Theta(s,x)G(s,y)\, ds
&
\leq 2 |\dv u|(x) \sup_{s\in\T} |G(s,y)|\,,
\end{align*}
for any bounded function $G$,
we deduce that
\begin{align*}
\int_\Omega I\,\eta^2\, dx
&
\lesssim
\|\varphi\|_{L^\infty(\T)} \int_\Omega
 \big(
|\dv u|+ T^{\tau\mathbf e_1} |\dv u|
\big) |D^{\tau\mathbf e_1}u| \,\eta^2\, dx
\\
&
\quad
+\|\varphi\|_{L^1(\T)}
\int_\Omega |\dv u|\, | D^{\tau\mathbf e_1}[\eta^2]|\, dx.
\end{align*}
Using Cauchy-Schwarz' inequality and the fact that $|D^{\tau\mathbf e_1}[\eta^2]|\leq 2|\tau| \|\nabla\eta\|_\infty$ since $|\eta|\leq 1$,
we infer
\begin{align*}
\int_\Omega I\,\eta^2\, dx
&
\lesssim
\|\varphi\|_{L^\infty(\T)}
\|\dv u\|_{L^2(\supp(\eta))}
\|\eta^2 |D^{\tau\mathbf e_1}u|\|_{L^2(\Omega)}
\\
&
\quad 
+\|\varphi\|_{L^\infty(\T)}
\|T^{\tau\mathbf e_1}|\dv u| \|_{L^2(\supp(\eta))}
\|\eta^2 |D^{\tau\mathbf e_1}u|\|_{L^2(\Omega)}
\\
&
\quad
+|\tau| \|\varphi\|_{L^1(\T)} \|\nabla\eta\|_\infty \|\dv u\|_{L^1(\Omega)},
\end{align*}
which implies the conclusion since
$\|T^{\tau\mathbf e_1}|\dv u| \|_{L^2(\supp(\eta))}\leq \|\dv u\|_{L^2(\Omega)}$.
\end{proof}

\subsection{Proof of the regularity estimate}
\label{ss:pf_reg_kin}

\begin{proof}[Proof of Theorem~\ref{t:reg_kin}]
Thanks to Lemma~\ref{l:comp_id},
for the functions $\chi$ and $\Theta$ given by Lemma~\ref{l:kin}, and any
odd, $\pi$-periodic $\varphi\in C^0(\T)$, we have the compensation identity \eqref{eq:comp_id}.
Integrating this compensation identity against $\eta^2\,dx$ for some $\eta\in C_c^1(\Omega;[0,1])$
and then with respect to $\tau$, we deduce
\begin{align*}
\int_\Omega \Delta^{\varphi,\chi}(x,\tau \mathbf e_1)\eta^2\, dx
&
=
\int_0^\tau 
\bigg(
\int_{\Omega} I^{\tau'}\eta^2\, dx \, d\tau'
-\int_{\Omega} \langle A^{\tau'},\nabla\eta\rangle \eta\, dx
\bigg) \, d\tau'\,.
\end{align*}
Thanks to Lemma~\ref{l:estim_I} we have
\begin{align*}
\int_{\Omega} I^{\tau'}\eta^2\, dx
&
\lesssim
|\tau' | 
\|\varphi\|_{L^1(\T)}
\|\nabla \eta\|_\infty 
\|\dv u\|_{L^1(\Omega)}
\\
&
\quad
+
\|\varphi\|_{L^\infty(\T)}
\|\dv u\|_{L^2(\Omega)}
\|\eta^2
 |D^{\tau'\mathbf e_1}u|
 \big\|_{L^2(\Omega)} 
\,,
\end{align*}
and thanks to Lemma~\ref{l:estim_A} we have
\begin{align*}
-\int_{\Omega} \langle A^{\tau'},\nabla\eta\rangle \eta\, dx
&
\lesssim
\|\varphi\|_{L^1(\T)} \|\nabla\eta\|_\infty \|\eta |D^{\tau'\mathbf e_1}u| \|_{L^1}\,.
\end{align*}
Plugging these  two inequalities into the previous identity
and using also that $\|\varphi\|_{L^1}\lesssim \|\varphi\|_{L^\infty}$, we obtain
\begin{align}
\int_\Omega \Delta^{\varphi,\chi}(x,\tau \mathbf e_1)\eta^2\, dx
&
\lesssim
\tau^2\|\varphi\|_{L^\infty}\|\nabla\eta\|_\infty\|\dv u\|_{L^1} 
\nonumber
\\
&
\quad
+  |\tau|
\|\varphi\|_{L^\infty}
\|\dv u\|_{L^2} \sup_{|\tau'|\leq |\tau|} \|\eta^2|D^{\tau'\mathbf e_1}u| \|_{L^2}
\nonumber
\\
&
\quad
+|\tau| \|\varphi\|_{L^\infty} \|\nabla\eta\|_\infty \sup_{|\tau'|\leq |\tau|} \|\eta |D^{\tau'\mathbf e_1}u| \|_{L^1}\,,
\label{eq:intDelta_first_estim}
\end{align}
for $|\tau|\leq \dist(\supp(\eta),\partial\Omega)$.
We have established this inequality for any continuous, odd and $\pi$-periodic function $\varphi$, 
but claim that it is also valid
 for any 
  odd, $\pi$-periodic function $\varphi\in L^\infty(\T)$.
Consider indeed a sequence of smooth, odd, $\pi$-periodic functions $\varphi_j\to\varphi$ a.e. on $\T$ and such that $\|\varphi_j\|_{L^\infty}\leq \|\varphi\|_{L^\infty}$ and apply \eqref{eq:intDelta_first_estim} to $\varphi_j$.
The right-hand side depends linearly on $\|\varphi_j\|_{L^\infty}\leq \|\varphi\|_{L^\infty}$,
and the left-hand side converges by dominated convergence, so we do conclude that \eqref{eq:intDelta_first_estim} is satisfied by the limit map $\varphi\in L^\infty(\T)$.
 
 Next we make a specific choice of such a function,
given by
\begin{align*}%\label{eq:varphi_alpha}
\varphi(t)=1 \quad \text{for }0<t\leq \frac\pi 2\,,
\end{align*}
and extended as an odd $\pi$-periodic function.
As recalled above \eqref{eq:Delta_varphi0_coerc},
for this choice of function $\varphi$,
the quantity $\Delta^{\varphi,\chi}$ satisfies
\begin{align*}%\label{eq:coercDeltavarphi}
\Delta^{\varphi,\chi}(x,h) \gtrsim |D^h u(x)|^{3}\,,
\end{align*}
hence \eqref{eq:intDelta_first_estim} implies
\begin{align*}
\int_\Omega |D^{\tau \mathbf e_1} u|^3 \eta^2\, dx
&
\lesssim
\tau^2\|\nabla\eta\|_\infty\|\dv u\|_{L^1} 
\\
&
\quad
+  |\tau|
\|\dv u\|_{L^2} \sup_{|\tau'|\leq |\tau|} \|\eta^2|D^{\tau'\mathbf e_1}u| \|_{L^2}
\\
&
\quad
+|\tau|\|\nabla\eta\|_\infty \sup_{|\tau'|\leq |\tau|} \|\eta |D^{\tau'\mathbf e_1}u| \|_{L^1}\,.
\end{align*}
Setting
\begin{align*}
X :=\sup_{|\tau'|\leq |\tau|}
\|\eta |D^{\tau'\mathbf e_1}u| \|_{L^{3}}\,,
\end{align*}
and using Hölder's inequality
and the fact that $|\eta|\leq 1$,
we deduce
\begin{align*}
X^{3}
&
\lesssim
\tau^2\|\nabla\eta\|_\infty
|\Omega|^{\frac 12}
\|\dv u\|_{L^2} 
%\\
%&
%\quad
+  |\tau|
\|\dv u\|_{L^2} |\Omega|^{\frac{1}{6}} X
%\\
%&
%\quad
+|\tau| \|\nabla\eta\|_\infty
 |\Omega|^{\frac{2}{3}} X
 \,.
\end{align*}
To estimate the last two terms we use Young's inequality
\begin{align*}
ab\leq \frac{2}{3}\lambda^{-\frac 12}a^{\frac{3}{2}} + \frac{\lambda}{3}b^{3}\qquad\forall a,b\geq 0\,,\lambda > 0\,.
\end{align*}
We 
obtain, 
for any $\lambda>0$,
\begin{align*}
X^{3}
&
\lesssim
\tau^2\|\nabla\eta\|_\infty
|\Omega|^{\frac 12}
\|\dv u\|_{L^2} 
%\\
%&
%\quad
+ \lambda^{-\frac 12}
|\tau|^{\frac{3}{2}}
\big(
\|\dv u\|_{L^2}^{\frac{3}{2}} |\Omega|^{\frac{1}{4}} 
+\|\nabla\eta\|_\infty^{\frac{3}{2}}
 |\Omega|
 \big)
 \\
 &
 \quad
+ 2\lambda X^{3}\,.
\end{align*}
Choosing $\lambda$ small enough, 
we infer
\begin{align*}
X^{3}
&
\lesssim
C\,(1+\|\dv u\|_{L^2})^{\frac{3}{2}}
|\tau|^{\frac{3}{2}}
\,,
\end{align*}
for some $C=C(\Omega,\eta)>0$.
Recall $\Omega'\subset\subset\Omega$.
Fixing $\eta\in C_c^1(\Omega;[0,1])$ such that $\eta=1$ on $\Omega'$, 
this gives exactly the estimate \eqref{eq:reg_estim},
that is,
\begin{align*}
\|D^h u\|_{L^3(\Omega'\cap(\Omega'-h))} \leq C (1+\|\dv u\|_{L^2(\Omega)})^{\frac{1}{2}}|h|^{\frac{1}{2}}\,,
\end{align*}
for $h=\tau \mathbf e_1$ such that 
$|h|\leq \delta_0$
for some $\delta_0=\delta_0(\Omega,\Omega')>0$.
That last restriction can be automatically removed, 
using simply for $|h|\geq \delta_0$ that $|D^h u|\leq 2$.
So this is valid for all $h=\tau \mathbf e_1$, $\tau\in\R$.
Finally, the direction $\mathbf e_1$ does not play a particular role, we can apply this same estimate to $\tilde u(x)=R^{-1}u(Rx)$ for any rotation $R$
(noting that this preserves the $L^2$ norm of the divergence)
 and deduce \eqref{eq:reg_estim} for $h=\tau R\mathbf e_1$,
hence for all $h\in\R^2$.
\end{proof}

\subsection{Proof of Corollary~\ref{c:jac}}
\label{ss:csq_reg_kin}

In this section we 
give the proof of Corollary~\ref{c:jac}
about the Jacobian.

\begin{proof}[Proof of Corollary~\ref{c:jac}]
Recall that the Jacobian distribution $Ju$ is well-defined and continuous on $W^{1/3,3}_{\loc}(\Omega;\mathbb S^1)$ thanks to \cite[Corollary~8.1]{BM21},
coincides with $Ju=(1/2)\curl(u\wedge \nabla u)$ for $u\in W^{1,1}_{\loc}\cap L^\infty$,
and with $Ju=\det(\nabla u)=0$ for $u\in H^1_{\loc}(\Omega;\mathbb S^1)$.

Let $u\in \mathcal S$, then there exists $u_\e\to u$ in $\mathcal D'(\Omega)$ such that $\sup E_\e(u_\e)<\infty$ along a sequence $\e\to 0$.
Thanks to Corollary~\ref{c:Ws3} we have $u_\e\to u$ in $W^{1/3,3}(\Omega')$
and therefore $Ju=\lim Ju_\e=0$ in $\mathcal D'(\Omega')$ for every $\Omega'\subset\subset\Omega$, hence $Ju=0$ in $\mathcal D'(\Omega)$.

 We also note that since
$B^{1/2}_{3,\infty,\loc}(\Omega)\subset W^{3/7,3}_{\loc}(\Omega)$,
we have
 $u\in W^{3/7,3}_{\loc}(\Omega;\mathbb S^1)$ thanks to Theorem~\ref{t:reg_estim},
so \cite[Corollary~14.3]{BM21} implies 
 the existence of
$v\in C^\infty(\Omega;\mathbb S^1)$ and 
$\varphi\in W^{2/5,3}_{\loc}(\Omega;\R)$ 
such that $u=e^{i\varphi}v$.
(Here we used that the regularity $W^{3/7,3}+W^{1,9/7}$ 
 given by \cite[Corollary~14.3]{BM21} for the function $\varphi$ implies $W^{2/5,3}$ regularity since $3/7>2/5$.)

Let $\Gamma\subset\Omega$ be a smooth curve homeomorphic to a circle.
Choose any $\delta>0$ such that the $\delta$-neighborhood 
$\omega_\delta =\lbrace x\in\R^2\colon \dist(x,\Gamma)\leq \delta\rbrace$ of $\Gamma$ is contained inside $\Omega$ 
and there is a $C^1$-diffeomorphism 
$\Theta\colon \Gamma\times [-\delta,\delta ]\to\omega_\delta$ such that 
$\Theta(\Gamma\times \lbrace 0\rbrace)=\Gamma$.
Then, letting $\Gamma_t=\Theta(\Gamma\times \lbrace t\rbrace)$,
 by Fubini's theorem we have $\varphi\in W^{2/5,3}(\Gamma_t;\R)\subset C^0(\Gamma_t;\R)$ for a.e. $t\in (-\delta,\delta)$.
 Hence there is a well-defined winding number 
 $\deg(u;\Gamma_t)=\deg(v;\Gamma_t)=\deg(v;\Gamma)$.
 This winding number being independent of $t$ and of the choice of the diffeomorphism $\Theta$, 
 this gives a well-defined meaning to $\deg(u;\Gamma)$.
 Moreover, for any sequence $u_\e\to u$ with $E_\e(u_\e)\leq C$,
 thanks to the $W^{2/5,3}_{\loc}$ convergence provided by Corollary~\ref{c:Ws3} 
 and invoking again  Fubini and the embedding $W^{2/5,3}(\Gamma_t)\subset C^0(\Gamma_t)$, 
 for a.e. $t\in (-\delta,\delta)$ we have uniform convergence of $u_\e$ on $\Gamma_t$ along a subsequence,
 hence $\deg(u_\e;\Gamma)=\deg(u_\e;\Gamma_t)\to \deg(u;\Gamma_t)=\deg(u;\Gamma)$.
\end{proof}

\section{Compactness of traces}\label{s:trace}

In this section we prove Theorem~\ref{t:trace}.

\medskip

Let $(u_j)\subset H^1(\Omega;\mathbb S^1)$ 
be such that $\sup_j\|\dv u_j\|_{L^2(\Omega)}<\infty$.
Then 
 from Lemma~\ref{l:kinlim}, we have, along a non-relabeled subsequence,
\begin{align*}
&
u_j\to u\qquad\text{in }L^1_{\loc}(\Omega)\text{ and }a.e.,
\\
&
\dv \Phi(u_j)\rightharpoonup \dv\Phi(u)=\lambda_\Phi(u)\dv u\quad
\text{weakly in }L^q(\Omega)
\text{ for }1 < q<2\,,
\end{align*}
for all  $\Phi\in\mathrm{ENT}$ \eqref{eq:ENT}.
In particular the limit map $u\in H_{\dv}(\Omega;\mathbb S^1)$ satisfies  the chain rule property \eqref{eq:chainENT} and therefore the kinetic formulation \eqref{eq:kin}.
The arguments in \cite{vasseur01} thus ensure
 that $u$ has a trace on $\partial\Omega$, attained in the strong $L^1$ sense.

\begin{proposition}
\label{p:tracelim}
If $u\in B^{1/2}_{3,\infty,\loc}(\Omega;\mathbb S^1)$ satisfies
$\dv u\in L^2(\Omega)$ and $\dv\Phi(u)=\lambda_\Phi(u)\dv u$ for all $\Phi\in\mathrm{ENT}$, 
then $u$ admits a strong trace on $\partial\Omega$, that is,
there exists a measurable map $\tr(u)\colon\partial\Omega\to\mathbb S^1$ such that
\begin{align}\label{eq:strong_trace}
\esslim_{\delta\to 0}\int_{\partial\Omega} |u(x-\delta\nu(x))-\tr(u)(x)|\, d\mathcal H^1(x)= 0\,,
\end{align}
where $\nu(x)$ denote the exterior unit normal to $x\in\partial\Omega$,
and $\esslim$ denotes the limit as $\delta\to 0$ 
with $\delta$ outside a negligible subset of $(0,\delta_0)$, 
for a small enough $\delta_0>0$.
\end{proposition} 

The proof of Proposition~\ref{p:tracelim} is a direct adaptation of the arguments in \cite{vasseur01}. 
For the readers' convenience, we sketch this adaptation in Appendix~\ref{a:traces}.
Next we establish that traces of bounded energy sequences converge strongly to the trace of their limit.

\begin{lemma}\label{l:limtrace}
If $(u_j)$ and $u$ are as in Lemma~\ref{l:kinlim}, 
then $\tr(u_j)\to\tr(u)$ in $L^1(\partial\Omega)$.
%Moreover, $u$ satisfies the kinetic formulation \eqref{eq:kin}.
\end{lemma}
 
\begin{proof}[Proof of Lemma~\ref{l:limtrace}]
We consider any Young measure 
$\lbrace\mu_x\rbrace\in \mathcal P(\R^2)^{\partial\Omega}$ generated by the bounded sequence $\lbrace\tr(u_j)\rbrace\subset L^\infty(\partial\Omega)$
and prove that
 $\mu_x$ is, for $\mathcal H^1$-almost every $x\in\partial\Omega$, a Dirac mass at $\tr(u)(x)$, which implies the strong convergence $\tr(u_j)\to\tr(u)$ in $L^2(\partial\Omega)$, see e.g.
 \cite[\S~1.5.3]{evans90}.
Here, recall that $\mu_x$ is characterized
by the convergence,
along a (nonrelabeled) subsequence $j\to\infty$,
\begin{align*}
\int_{\partial\Omega} \zeta\, F(\tr(u_j))\, d\mathcal H^1 \to \int_{\partial\Omega}\zeta(x)\, \int_{\R^2}F(y)\, d\mu_x(y)\, d\mathcal H^1(x)\,,
\end{align*}
for all 
$F\in C^0(\R^2)$ and  $\zeta\in L^1(\partial\Omega)$.
Thanks to Proposition~\ref{p:tracelim}, $u$ has a strong trace $\tr(u)$ on $\partial\Omega$,
see \eqref{eq:strong_trace}, and it satisfies the Gauss-Green identity
\begin{align*}
\int_\Omega \eta \, \dv \Phi(u) \, dx
+\int_\Omega \langle \nabla\eta,\Phi(u) \rangle\, dx 
=\int_{\partial\Omega} \eta\, \langle \nu,\Phi(\tr(u))\rangle\, d\mathcal H^1\,,
\end{align*} 
for all $\eta\in C_c^1(\R^2)$, see  \cite{CF99}.
%Here, recall that $\nu(x)$ denote the exterior unit normal to $x\in\partial\Omega$.
We can also apply this formula to $u_j\in H^1(\Omega;\mathbb S^1)$.
Using 
that $\dv\Phi(u_j)\rightharpoonup \dv\Phi(u)$ weakly in $L^{3/2}(\Omega)$ and $\Phi(u_j)\to\Phi(u)$ strongly in $L^1(\Omega)$, we deduce
\begin{align*}
\int_{\partial\Omega}\eta \langle \nu,\Phi(\tr(u_j))\rangle\, d\mathcal H^1
\to
\int_{\partial\Omega}\eta \langle \nu,\Phi(\tr(u))\rangle\, d\mathcal H^1\,,
\end{align*}
for all $\eta\in C_c^1(\R^2)$.
By definition of the Young measure $\mu_x$, this implies
\begin{align*}
\int_{\partial\Omega}\eta(x)
 \big\langle \nu(x),
 \int_{\R^2} 
\widetilde \Phi(y) 
d\mu_x(y)
-\Phi(\tr(u)(x))
\big\rangle\,
d\mathcal H^1(x)
=0,
\end{align*}
for all $\eta\in C_c^1(\R^2)$ and $\Phi\in\mathrm{ENT}$,
and any extension $\widetilde\Phi\in C_c^1(\R^2;\R^2)$ such that 
$\widetilde\Phi_{|\mathbb S^1}=\Phi$.
This being valid for all $\eta$, we infer
\begin{align}\label{eq:muxENT}
 \big\langle \nu(x),
 \int_{\R^2} 
\widetilde\Phi(y) 
d\mu_x(y)
-\Phi(\tr(u)(x))
\big\rangle\,
=0\,,
\qquad\text{for a.e. }x\in\partial\Omega\,.
\end{align}
We fix $x\in\partial\Omega$ such that this identity is satisfied.
First note that we may add to the extension 
$\widetilde \Phi$ any map 
$\Psi\in C_c^1(\R^2)$ such that $\Psi_{|\mathbb S^1}=0$. 
This implies that
\begin{align*}
 \big\langle \nu(x),
 \int_{\R^2} 
\Psi(y) 
d\mu_x(y)
\big\rangle\,
=0\,,
\qquad\forall \Psi\in C_c^1(\R^2)\text{ with }\Psi_{|\mathbb S^1}=0\,.
\end{align*}
As a consequence,
 $\mu_x$ is supported in $\mathbb S^1$.
Then, applying \eqref{eq:muxENT}
 to the entropies $\Phi_g$ given in \eqref{eq:Phi_g}, we obtain
\begin{align*}
\int_{\T} g(s)\, \langle e^{is}, \nu(x) \rangle
\bigg( \int_{\mathbb S^1} 
\mathbf 1_{ e^{is}\cdot y >0}
\,
d\mu_x(y) 
-\mathbf 1_{e^{is}\cdot \tr(u)(x) >0}
\bigg)\, ds
=0\,,
\end{align*}
for all $g\in C^0(\T)$, and therefore, 
letting $\xi=\tr(u)(x)\in\mathbb S^1$,
\begin{align*}
 \int_{\mathbb S^1} 
\mathbf 1_{e^{is}\cdot y >0}
\,
d\mu_x(y)
=
\mathbf 1_{e^{is}\cdot \xi > 0}
\qquad\text{for a.e. }s\in\T\,,
\end{align*}
so the semi-circle 
$C_s^+ =
\lbrace y\in\mathbb S^1\colon e^{is}\cdot y >0\rbrace$
satisfies
\begin{align*}
\mu_x(C_s^+)
=0
\qquad\text{for a.e. }s\in\T_\xi^- =\lbrace s\in\T\colon e^{is}\cdot\xi <0\rbrace\,.
\end{align*}
Applying this to a dense sequence 
$\lbrace s_k\rbrace\subset\T_\xi^-$  we deduce that
\begin{align*}
\mu_x(\mathbb S^1\setminus\lbrace\xi\rbrace)
=\mu_x\Big(
\bigcup_{k} C_{s_k}^+
\Big)
=0\,.
\end{align*}
We conclude that the Young measure $\mu_x$ is a Dirac mass at 
$\xi=\tr(u)(x)$ for a.e. $x\in\partial\Omega$, 
and therefore $\tr(u_j)\to\tr(u)$ strongly in $L^1(\partial\Omega)$.
\end{proof}

\begin{proof}[Proof of Theorem~\ref{t:trace}]
Theorem~\ref{t:trace} follows directly from Lemma~\ref{l:kinlim}, Proposition~\ref{p:tracelim} and Lemma~\ref{l:limtrace}.
\end{proof}

\section{Recovery sequences}\label{s:recov}

In this section we prove Theorem~\ref{t:recovery2D}.
It relies on the properties of the convolution
\begin{align}\label{eq:udelta}
u_\delta =u *\rho_\delta\qquad\text{in }\Omega_\delta=\lbrace x\in\Omega\colon\dist(x,\partial\Omega)>\delta\rbrace\,,
\end{align}
where $\rho_\delta(x)=\delta^{-2}\rho(x/\delta)$ and $\rho\in C_c^1(B_1;[0,1])$ satisfies $|\nabla\rho|\leq 2$ and $\int_{B_1}\rho\, dx =1$.
We will obtain the recovery sequence by projecting $u_\delta$ onto $\mathbb S^1$, namely setting $v_\delta=u_\delta/|u_\delta|$,
and then modifying this map near the boundary in order
to extend it to
 the whole $\Omega$ instead of the smaller domain $\Omega_\delta$.
The main step
is to establish
convergence of $v_\delta$ and $\dv v_\delta$
in this smaller domain.
This is done in the following lemma,
 by relying 
on commutator estimates which provide
an improved control on the oscillations of
$|u_\delta|^2$
thanks to the 
commutator structure $1-|u_\delta|^2=|u|^2_\delta -|u_\delta|^2$.

\begin{proposition}\label{p:suffB1/2_4c0}
Assume $u\colon  \Omega\to\mathbb S^1$ is
such that $\dv u\in L^2(\Omega)$,  and  it satisfies \eqref{eq:suffcondrecov}, that is,
\begin{align*}
 \Xint{-}_{B_\delta} 
 \bigg(
 \frac{\|D^h u\|_{L^4(\Omega\cap(\Omega-h))}}
 {\delta^{1/2}}
 \bigg)^4 \, dh 
 \longrightarrow 0
\qquad\text{ as }\delta\to 0\,.
\end{align*}
Then, for small enough $\delta>0$, the convolution $u_\delta$ defined in \eqref{eq:udelta} does not vanish in $\Omega_\delta$, 
and the map $v_\delta=u_\delta/|u_\delta|\in C^1(\Omega_\delta;\mathbb S^1)$ satisfies
 \begin{align*}
 \int_{\Omega_\delta} |v_\delta - u|^2\, dx \to 0
 \quad\text{and}
 \quad
 \int_{\Omega_{\delta}} (\dv v_\delta - \dv u)^2\, dx \to 0\,,
 \end{align*}
 as $\delta\to 0$.
\end{proposition}

Proposition~\ref{p:suffB1/2_4c0} relies on the chain rule
 and on the properties of $|u_\delta|$, namely:

\begin{lemma}\label{l:udelta}
Assume $u\colon  \Omega\to\mathbb S^1$ satisfies \eqref{eq:suffcondrecov}, then the convolution $u_\delta$ defined in \eqref{eq:udelta} satisfies
\begin{align*}
\sup_{\Omega_{\delta}} \big|1-|u_\delta|^2\big|
\to 0
\quad\text{and}
\quad
\int_{\Omega_\delta}|\nabla [|u_\delta|^2]|^2\, dx  \to 0\,,
\end{align*}
as $\delta\to 0$.
\end{lemma}
\begin{proof}[Proof of Lemma~\ref{l:udelta}]
One approach to establish the first property 
is to notice that $u$ is VMO thanks to the assumption \eqref{eq:suffcondrecov},
and this is well-known to imply the uniform convergence of $\dist(u_\delta,\mathbb S^1)$ to zero, see e.g. % \cite[Appendix]{BMBGP91} or
\cite[eqn.(7)]{BN95}).
The second property will follow from commutator estimates (as used e.g. in \cite{CET94,DLI15})
 for
$1-|u_\delta|^2=|u|^2 * \rho_\delta - |u * \rho_\delta|^2$.
We present here a self-contained proof of both properties relying on these commutator estimates.

We start from the identity
\begin{align*}
1-|u_\delta(x)|^2
&
=|u(x- y)|^2 -|u_\delta(x)|^2
\\
&
=|u(x- y)-u_\delta(x)|^2 + 2 \langle u_\delta(x),u(x- y)-u_\delta(x)\rangle\,,
\end{align*}
valid for all $x\in\Omega_\delta$ and $y\in B_\delta$. Integrating with respect to $\rho_\delta(y)\,dy$, 
the last term integrates to zero and we are left with
\begin{align}\label{eq:commut_id}
1-|u_\delta(x)|^2
&
=\int_{B_\delta} |u(x- y)-u_\delta(x)|^2\, \rho_\delta(y)\,dy
\nonumber
\\
&
=\int_{\Omega} \bigg| \int_{\Omega} \big(u(y)-u( z)\big)
\rho_\delta(x-z)\,dz
\bigg|^2\rho_\delta(x-y)\,dy
\end{align}
As a first consequence of \eqref{eq:commut_id},
 we deduce, using Jensen's inequality and the fact that $\rho_\delta\leq \delta^{-2}$,
\begin{align*}
\big|1-|u_\delta(x)|^2\big|^2
&
\leq 
\int_{\Omega} \int_{\Omega} \big|
u(y)-u( z)\big|^4
\rho_\delta(x-y)\,dy
\,
\rho_\delta(x-z)\,dz
\\
&
=
\int_{B_\delta(x)}\! 
\int_{B_\delta(x-z)}
\!\!
 \big|
u(z+h)-u( z)\big|^4
\rho_\delta(x-z-h)\,dh
\,
\rho_\delta(x-z)\,dz
\\
&
\leq
\frac{1}{\delta^4}
\int_{B_\delta(x)}
\int_{B_{2\delta}} 
\mathbf 1_{z+h\in\Omega} 
\big|
D^h u(z)\big|^4
\,dh
\,dz
\\
&
\leq \int_{B_{2\delta}}\frac{1}{\delta^4}\|D^h u\|^4_{L^4(\Omega\cap(\Omega-h))}\,dh
\\
&
=\pi \Xint{-}_{B_{2\delta}} \bigg(\frac{\|D^h u\|_{L^4(\Omega\cap(\Omega-h))} }{\delta^{1/2}} 
\bigg)^4\, dh\,.
\end{align*}
This last quantity is independent of $x\in\Omega_\delta$,
 and tends to zero by assumption \eqref{eq:suffcondrecov}, 
hence the first convergence property.

To establish the convergence of $\nabla [|u_\delta|^2]$ in $L^2$, we differentiate \eqref{eq:commut_id} and obtain
\begin{align*}
-\nabla[|u_\delta|^2](x)
&
=\frac{1}{\delta}\int_{\Omega} \bigg| \int_{\Omega} \big(u(y)-u( z)\big)
\rho_\delta(x-z)\,dz
\bigg|^2(\nabla\rho)_\delta(x-y)\,dy
\\
&
\quad
+\frac 2\delta \int_{\Omega}
\Big\langle \int_{\Omega}\big(u(y)-u( z')\big)
\rho_\delta(x-z')\,dz'
,
\\
&
\hspace{5em}
\int_{\Omega}\big(u(y)-u( z)\big)
(\nabla\rho)_\delta(x-z)\,dz
\Big\rangle
\,\rho_\delta(x-y)\, dy\,.
\end{align*}
Recalling that $|\nabla \rho|\leq 2$ and using again Jensen's inequality this implies
\begin{align*}
\big|\nabla[|u_\delta|^2]\big|^2(x)
&
\lesssim \frac{1}{\delta^2}\Xint{-}_{B_{\delta}(x)}\Xint{-}_{B_\delta(x)}|u(y)-u(z)|^4\, dy\, dz
\\
&
= \frac{1}{\delta^2}\Xint{-}_{B_{\delta}}\Xint{-}_{B_\delta}|u(x+h)-u(x+k)|^4\, dh\, dk
\\
&
\lesssim \frac{1}{\delta^2}
\Xint{-}_{B_\delta}|u(x+h)-u(x)|^4\, dh
=\Xint{-}_{B_{\delta}}\frac{|D^h u(x)|^4}{\delta^{2}} 
\, dh
\end{align*}
Integrating this inequality we infer
\begin{align*}
\int_{\Omega_\delta}\big|\nabla [|u_\delta|^2]\big|^2\, dx \lesssim
\Xint{-}_{B_{\delta}} \bigg(\frac{\|D^h u\|_{L^4(\Omega\cap(\Omega-h))} }{\delta^{1/2}} 
\bigg)^4\, dh\,,
\end{align*}
and this tends to zero by assumption \eqref{eq:suffcondrecov}.
\end{proof}

\begin{proof}[Proof of Proposition~\ref{p:suffB1/2_4c0}]
Thanks to the first convergence property in Lemma~\ref{l:udelta}, 
for small enough $\delta>0$ we have $|u_\delta|\geq 1/2$
in $\Omega_\delta$.
Hence the map $v_\delta=u_\delta/|u_\delta|$ is well defined and $C^1$ in $\Omega_\delta$.
The map $\xi\mapsto \xi/|\xi|$ is Lipschitz on $\lbrace |\xi|\geq 1/2\rbrace$, so we have
\begin{align*}
\int_{\Omega_\delta}|v_\delta - u|^2\, dx  \lesssim \int_{\Omega_\delta}|u_\delta - u|^2\,dx\,,
\end{align*}
and this last integral tends to zero
thanks to classical properties of  convolutions.
Moreover, using the chain rule we compute
\begin{align*}
\dv  v_\delta 
&
=\frac{\dv u_\delta}{|u_\delta|}-\frac{1}{2|u_\delta|^3}
\langle u_\delta ,  \nabla [ |u_\delta|^2] \rangle\,.
\end{align*}
Since $\dv u_\delta =(\dv u)_\delta$,
this implies
\begin{align*}
\dv v_\delta -\dv u
&
=(\dv u)_\delta - \dv u 
+(\dv u)_\delta \frac{|u_\delta|-1}{|u_\delta|}
-\frac{1}{2|u_\delta|^3}
\langle u_\delta ,  \nabla [ |u_\delta|^2] \rangle\,,
\end{align*}
and therefore
\begin{align*}
\|\dv v_\delta -\dv u\|_{L^2(\Omega_\delta)}
&
\leq 
\| (\dv u)_\delta -\dv u\|_{L^2(\Omega_\delta)}
\\
&\quad
+2\|\dv u \|_{L^2(\Omega)} \sup_{\Omega_\delta}|1-|u_\delta||
\\
&
\quad
 +2 \|\nabla [|u_\delta|^2]\|_{L^2(\Omega_\delta)}\,.
\end{align*}
The first term in the right-hand side tends to 0 by classical properties of convolutions, and
the last two terms tend to zero thanks to Lemma~\ref{l:udelta}.
\end{proof}

In order to prove Theorem~\ref{t:recovery2D} from Proposition~\ref{p:suffB1/2_4c0}, 
we will need to modify the map $v_\delta$ defined on the slightly smaller domain $\Omega_\delta$,
into a map $\tilde v_\delta$ defined on the full domain $\Omega$.
This will require some control on $v_\delta$ in $H^1$, which we establish before showing the proof of Theorem~\ref{t:recovery2D}.

\begin{lemma}\label{l:vdeltaH1}
Let $\Omega\subset\R^2$ a bounded open set and
assume $u\colon  \Omega\to\mathbb S^1$ satisfies \eqref{eq:suffcondrecov}.
Then the map $v_\delta\in C^1(\Omega_\delta;\mathbb S^1)$, defined in Proposition~\ref{p:suffB1/2_4c0} for small $\delta >0$,
satisfies
\begin{align*}
\delta\int_{\Omega_\delta}|\nabla v_\delta|^2\, dx \to 0\qquad\text{as }\delta\to 0\,,
\end{align*}
where $\Omega_\delta$ is defined in \eqref{eq:udelta}.
\end{lemma}
\begin{proof}[Proof of Lemma~\ref{l:vdeltaH1}]
Recall that $v_\delta=u_\delta/|u_\delta|$ and $|u_\delta|\geq 1/2$ for small $\delta >0$.
The map $\xi\mapsto \xi/|\xi|$ is Lipschitz on $\lbrace |\xi|\geq 1/2\rbrace$,
so it suffices to establish the estimate for $u_\delta=u * \rho_\delta$, that is,
\begin{align*}
\delta\int_{\Omega_\delta}|\nabla u_\delta|^2\, dx \to 0\qquad\text{as }\delta\to 0\,,
\end{align*}
For all $x\in\Omega_\delta$ we have
\begin{align*}
\nabla u_\delta(x)
&
=\int_\Omega u(y) \nabla\rho_\delta (x-y)\, dy 
%\\
%&
=\frac 1\delta \int_{B_1} u(x-\delta y)\nabla\rho(y)\, dy
\\
&
=\frac 1\delta \int_{B_1} D^{-\delta y}u(x)\, \nabla\rho(y)\, dy\,.
\end{align*}
The last equality is obtained by subtracting $u(x)$ in the integral with respect to the zero-average measure $\nabla\rho(y)\, dy$.
Squaring, integrating, and applying Jensen's inequality, we deduce
\begin{align*}
\bigg(
\delta \int_{\Omega_\delta }
|\nabla u_\delta|^2\, dx
\bigg)^2
&
\lesssim |\Omega_\delta| 
\int_{B_1}\int_{\Omega_\delta }
\frac{|D^{-\delta y} u|^4}{\delta^2}\, dx \, dy
\\
&
\lesssim 
 |\Omega_\delta|
\Xint{-}_{B_\delta} 
\bigg(\frac{\|D^h u\|_{L^4(\Omega\cap(\Omega-h))}}{\delta^{1/2}}\bigg)^4 \, dh\,,
\end{align*}
and this last expression tends to zero by assumption \eqref{eq:suffcondrecov}.
\end{proof}

Finally we use Proposition~\ref{p:suffB1/2_4c0} to prove Theorem~\ref{t:recovery2D}.

\begin{proof}[Proof of Theorem~\ref{t:recovery2D}]
We first need to adjust the maps $v_\delta$ provided by Proposition~\ref{p:suffB1/2_4c0} in the slightly smaller domain $\Omega_\delta$ in order to obain approximating maps in the whole domain $\Omega$.
This is done by performing 
a small dilation in a neighborhood of $\partial\Omega$.

Denote by $\nu\colon \partial\Omega\to\mathbb S^1$ the outer unit normal to $\partial\Omega$, and fix $\delta_0>0$ such that
\begin{align*}
\Xi\colon \partial\Omega\times (0,\delta_0) &
\to \Omega\setminus\overline \Omega_{\delta_0} = \lbrace x\in\Omega\colon \dist(x,\partial\Omega)<\delta_0\rbrace
\\
(x,t) &\mapsto 
x-t\nu(x)\,,
\end{align*}
is a $C^2$ diffeomorphism.
We use it to define, for $0<\delta <\delta_0$, a $C^2$ diffeomorphism
\begin{align*}
\Psi_\delta\colon \Omega\setminus \Omega_{\delta_0}\to\Omega_\delta\setminus \Omega_{\delta_0}\,,
\quad
\Psi_\delta =\Xi\circ\widetilde\Psi_\delta\circ\Xi^{-1}\,,
\end{align*}
where $\widetilde\Psi_\delta =\Xi^{-1}\circ \Psi_\delta \circ \Xi\colon\partial\Omega\times (0,\delta_0)\to \partial\Omega\times (\delta,\delta_0)$
is given by
\begin{align*}
\widetilde\Psi_\delta(x,t)
=\big(x,\delta + (1-\delta/\delta_0) t  \big)
= (x,t) + \frac{\delta}{\delta_0}(0,\delta_0 - t)\,.
\end{align*}
Using that $|\widetilde\Psi_\delta - \I_{\partial\Omega\times\R}|\leq \delta$ and $|D\widetilde\Psi_\delta - \I_{T(\partial\Omega)\times\R} |\leq \delta/\delta_0$ on $\partial\Omega\times (0,\delta_0)$,
we find that
\begin{align*}
|D\Psi_\delta - \I_{\R^2}|\lesssim \delta \bigg( \frac{\|D\Xi\|_\infty}{\delta_0} +   \mathrm{Lip}(D\Xi)
\bigg) 
\|D\Xi^{-1}\|_\infty  
\quad
\text{on }\Omega\setminus\Omega_{\delta_0}\,.
\end{align*}
Then we define $\tilde v_\delta\in H^1(\Omega;\mathbb S^1)$ by setting
\begin{align*}
\tilde v_\delta =\begin{cases}
v_\delta
&\quad\text{in }\Omega_{\delta_0}\,,
\\
v_\delta\circ\Psi_\delta
&
\quad\text{in }\Omega\setminus \Omega_{\delta_0}\,.
\end{cases}
\end{align*}
Using the above estimate on $D\Psi_\delta$ we see that
\begin{align*}
|\dv\tilde v_\delta -(\dv v_\delta ) \circ\Psi_\delta |
&
=\Big| 
\tr\big(Dv_\delta(\Psi_\delta) (D\Psi_\delta -\I_{\R^2}) \big)
\Big|
\\
&
\leq c \,\delta |\nabla v_\delta| \circ\Psi_\delta
\qquad\text{on }\Omega\setminus\Omega_{\delta_0} \,,
\\
\text{and }
|\det(D\Psi_\delta^{-1})-1|
&\leq c\, \delta 
\qquad\text{on }\Omega_\delta\setminus\Omega_{\delta_0}\,,
\end{align*}
for some $c>0$ depending on $\Omega$,
and therefore
\begin{align*}
\bigg|\int_{\Omega\setminus \Omega_{\delta_0}}(\dv \tilde v_\delta)^2\, dx 
-\int_{\Omega_\delta\setminus\Omega_{\delta_0}}(\dv v_\delta)^2\, dx
\bigg|
\lesssim c\delta \int_{\Omega_\delta\setminus \Omega_{\delta_0}}|\nabla v_\delta|^2\, dx
\end{align*}
Thanks to Lemma~\ref{l:vdeltaH1}, this last quantity tends to zero as $\delta\to 0$.
Invoking  Proposition~\ref{p:suffB1/2_4c0}
and using that $|\Omega\setminus\Omega_\delta|\to 0$, we 
are therefore able to conclude that
\begin{align*}
\int_{\Omega}(\dv \tilde v_\delta)^2\, dx 
\to \int_{\Omega}(\dv u)^2\, dx\,.
\end{align*}
Moreover, 
the bi-Lipschitzness of $\Psi_\delta$ and Lemma~\ref{l:vdeltaH1} ensure that
\begin{align*}
\int_{\Omega}|\nabla\tilde v_\delta|^2\, dx \lesssim \int_{\Omega_\delta}|\nabla v_\delta|^2\, dx
=o(1/\delta)\,,
\end{align*}
as $\delta\to 0$,
hence
 we see that 
\begin{align*}
E_\e(\tilde v_\delta) -\int_{\Omega}(\dv \tilde v_\delta)^2\, dx 
=
\e\int_{\Omega}(\curl\tilde v_\delta)^2\, dx
=o(\e/\delta)\,.
\end{align*}
Therefore, choosing $\delta=\delta_\e =\e$ we conclude that
\begin{align*}
E_\e(\tilde v_{\delta}) \to \int_{\Omega}(\dv u)^2\, dx\,.
\end{align*}
Finally we also check that $\tilde v_\delta\to u$ in $L^2(\Omega)$.
This follows from the facts that
 $v_\delta\to u$ in $L^2(\Omega_{\delta_0})$,
  $u\circ\Psi_\delta \to u$ in $L^2(\Omega\setminus\Omega_{\delta_0})$
 by dominated convergence, 
 and
\begin{align*}
\int_{\Omega\setminus \Omega_{\delta_0}} |\tilde v_\delta -u\circ \Psi_\delta|^2\, dx
&
\leq (1+c\,\delta) \int_{\Omega_\delta\setminus\Omega_{\delta_0}}|v_\delta - u|^2 \, dy \to 0\,,
\end{align*}
where we used $|\det(D\psi_\delta^{-1})|\leq 1+c\, \delta$.
\end{proof}

\section{The one-dimensional case}\label{s:1D}

In this section we prove Theorem~\ref{t:1D}. 
%Translating and rescaling, we assume without loss of generality that $I=(-1,1)$.
%
We start with a sharp version of the compensated regularity estimates, 
which can be obtained rather easily in this one-dimensional case.

\begin{proposition}\label{p:reg1D}
We have
\begin{align}\label{e:reg1D}
|u|_{B^{1/2}_{4,\infty}(I)}^4
\lesssim 
\int_I (u_1')^2\, dx\,,
\end{align}
for all $u\in H^1(I;\mathbb S^1)$.
\end{proposition}
\begin{proof}[Proof of Proposition~\ref{p:reg1D}]
Let $u\in H^1(I;\mathbb S^1)$ and $\varphi\in H^1(I;\R)$ such that $u=e^{i\varphi}$,
and note that $|u_1'|=|(\cos\varphi)'|
= |\sin\varphi| \, |\varphi'|=|F(\varphi)'|$, where
 $F\in C^1(\R;\R)$  is  given by 
\begin{align}\label{eq:F}
F(t)=\int_0^t |\sin s|\, ds\qquad\forall t\in\R\,.
\end{align}
Note also that $F$ is bijective and its inverse $F^{-1}$ is $C^{1/2}$. 
Explicitly, $F$ and $F^{-1}$ are characterized by
\begin{align*}
&F(t)=1-\cos t \text{ for }t\in [0,\pi],\\
&F(t +\ell\pi)=2\ell +F(t)\text{ for }t\in\R,\ell\in\Z,\\
&F^{-1}(s)
=\arccos (1-s)\text{ for }s\in [0,2],
\\
&
F^{-1}(s+2\ell)=\ell\pi +F^{-1}(s)\text{ for }s\in\R,\ell\in\Z,
\end{align*}
where $\arccos\colon [-1,1]\to [0,\pi]$ is the left inverse of $\cos$ on $[0,\pi]$.
Since the function $F^{-1}$ is $C^{1/2}$, we have, letting $w=F(\varphi)$ and $I_h=I\cap (I-h)$,
\begin{align*}
\int_{I_h} |\varphi(x+h)-\varphi(x)|^4\, dx
&
\lesssim \int_{I_h} |w(x+h)-w(x)|^2\, dx 
\\
&
\lesssim  |h|^2\int_I (w')^2\, dx =|h|^2 \int_I (u_1')^2\, dx,
\end{align*}
which implies \eqref{e:reg1D} since $u$ is a Lipschitz function of $\varphi$.
\end{proof}

Proposition~\ref{p:reg1D} implies the boundedness and compactness statement in Theorem~\ref{t:1D}.
The lower bound of the $\Gamma$-convergence statement also follows
rather directly:
 any sequence $(u_\e)\subset H^1(I;\mathbb S^1)$ such that
$\sup E_\e^{\mathrm{1D}}(u_\e)<\infty$ and $u_\e\to u$ in $\mathcal D'(\Omega)$ is bounded in $B^{1/2}_{4,\infty}(I;\mathbb S^1)$,
so $u$ must belong to that space, and the simple inequality $E_\e^{\mathrm{1D}}(u_\e)\geq E_0^{\mathrm{1D}}(u_\e)$ implies that $(u_{\e 1})$ is bounded, hence weakly convergent, in $H^1(I)$
and therefore
\begin{align*}
\liminf_{\e\to 0} E_\e^{\mathrm{1D}}(u_\e) \geq E_0^{\mathrm{1D}}(u)\,,
\end{align*}
thanks to the weak lower semicontinuity of $v\mapsto \int_I (v')^2\,dx$ on $H^1(I)$. 
Hence it only remains to prove the upper bound part of the $\Gamma$-convergence.

\begin{proposition}\label{p:recov1D}
Let $u\in B^{1/2}_{4,\infty}(I;\mathbb S^1)$ such that $u_1\in H^1(I)$. Then there exists $(u_\e)\subset H^1(I;\mathbb S^1)$ such that $u_\e\to u$ in $\mathcal D'(I)$ and $E_\e^{\mathrm{1D}}(u_\e)\to E_0(u)$ as $\e\to 0$.
\end{proposition}
\begin{proof}
Since  $B^{1/2}_{4,\infty}(I)\subset C^0(\overline I)$, 
the continuous $\mathbb S^1$-valued map $u$ admits a lifting $\varphi\in C^0(\overline I;\R)$ such that $u=e^{i\varphi}$.
We extend $\varphi$ to a continuous function $\varphi\colon\R\to\R$ which is smooth with compact support outside $I$ and constant on $[a-1,a]$ and $[b,b+1]$, where $I=(a,b)$. 
This also provides an extension of $u=e^{i\varphi}$  to $\R$.

Locally (that is, on any small interval where the image of $u$ is contained in an arc of length less than $\pi$),
the lifting $\varphi$ is given locally by $\varphi =\Theta(u)$ where $\Theta$ is a smooth left inverse of $t\mapsto e^{it}$ on an open interval of size $\pi$
(that is,  $\Theta(e^{it})=t$ for $t$ in that interval).
This implies that $\varphi$ inherits the regularity of $u$, 
hence we have
 $\varphi\in B^{1/2}_{4,\infty}(\R)$.
Moreover we have
$u_1=\cos\varphi\in H^1_{\loc}(\R)$, 
which implies $F(\varphi)\in H^1_{\loc}(\R)$, where $F$ is defined in \eqref{eq:F}.

In what follows,
we denote 
 convolution with a smooth compactly supported kernel $\rho_\delta(x)=\delta^{-1}\rho(x/\delta)$, $\rho\in C_c^\infty(-1,1)$, $\int\rho=1$ by a subscript $\delta$.
 
According to this notation, we 
consider the mollified phase function  $\varphi_\delta =\varphi * \rho_\delta$,
and define $u^{[\delta]} =e^{i\varphi_\delta}\in C_c^1(\R;\mathbb S^1)$.
Then we have $u^{[\delta]}\to e^{i\varphi}=u$ in $\mathcal D'(\R)$ as $\delta\to 0$,
and we claim that
\begin{align}\label{eq:recov1Dudelta}
\int_I |(u^{[\delta]}_1)'|^2\, dx \to \int_I (u_1')^2\, dx  = E_0(u).
\end{align}
Granted this, 
noting that
\begin{align*}
\varphi_\delta'(x)=\frac{1}{\delta}\int (\varphi(x-\delta y)-\varphi(x))\rho'(y)\, dy,
\end{align*}
implies
\begin{align*}
\int_I (\varphi_\delta')^2\, dx
&
\leq \frac{\|\rho'\|_{L^1}}{\delta^2}\int \|\varphi(\cdot -\delta y)-\varphi\|_{L^2}^2 |\rho'(y)|\, dy 
\\
&\lesssim
\frac{ \|\rho'\|_{L^1}}{\delta^2}\int \|\varphi(\cdot -\delta y)-\varphi\|_{L^4}^2 |\rho'(y)|\, dy,
\end{align*}
and therefore
\begin{align*}
\int_I |(u^{[\delta]}_2)'|^2\, dx  \leq \int_I
(\varphi_\delta')^2\, dx \lesssim \frac{1}{\delta}\|\varphi\|^2_{B^{1/2}_{4,\infty}},
\end{align*}
we see that choosing $\delta=\delta_\e =\sqrt\e$ 
ensures the recovery sequence property
\begin{align*}
E_\e(u^{[\delta]})\to \int_{I} (u_1')^2\, dx,
\end{align*}
which proves Proposition~\ref{p:recov1D}.

Thus it all boils down to showing \eqref{eq:recov1Dudelta}, 
which
 is equivalent to $F(\varphi_\delta)'\to F(\varphi)'$ in $L^2(I)$
 since $|(u^{[\delta]}_1)'|=|F(\varphi_\delta)'|$, and we devote the remainder of the proof to establishing this convergence.
 
 To that end
 we compute
\begin{align*}
&
\frac{d}{dx}
\left[
F(\varphi)_\delta (x)
-
F(\varphi_\delta(x))
\right]
\\
&
=
\int F(\varphi(y))\rho_\delta'(x-y)\, dy
-
F'(\varphi_\delta(x))\varphi_\delta'(x)
\\
&
=\int\Big(F(\varphi(y))- F'(\varphi_\delta(x))\varphi(y) \Big)\rho_\delta'(x-y)\, dy
\\
&
=\int\Big(
F(\varphi(y))-F(\varphi_\delta(x))
- F'(\varphi_\delta(x))(\varphi(y)-\varphi_\delta(x))
\Big)\rho_\delta'(x-y)\, dy
\\
&
=\int R(\delta,x,y) (\varphi(y)-\varphi_\delta(x))^2\rho_\delta'(x-y)\, dy,
\end{align*}
where
\begin{align*}
R(\delta,x,y)
&
=\int_0^1 (1-t)F''(\varphi_\delta(x) +t(\varphi(y)-\varphi_\delta(x)))\,dt.
\end{align*}
Since $|R|\leq \sup |F''| \leq 1$, we deduce
\begin{align*}
&
\left| F(\varphi_\delta)'-[F(\varphi)_\delta]'\right|
%\\
%&
\leq
\frac{1}{\delta}
\int (\varphi(x-\delta y)-\varphi(x-\delta z))^2\, \rho(z)\, dz \, |\rho'(y)|\,dy,
\end{align*}
and by Jensen's (or Cauchy-Schwarz') inequality,
\begin{align}
&
\left(F(\varphi_\delta)'-[F(\varphi)_\delta]'\right)^2
\nonumber
\\
&
\leq
\frac{\|\rho'\|_{L^1}}{\delta^2}
\int (\varphi(x-\delta y)-\varphi(x-\delta z))^4\, \rho(z)\, dz \, |\rho'(y)|\,dy.
\label{eq:estimFvarphidelta}
\end{align}
Integrating this on $I$ gives a bounded right-hand side
 thanks to the fact that $\varphi\in B^{1/2}_{4,\infty}(\R)$,
 but we would really like to obtain 
 a bound that tends to $0$ as $\delta\to 0$, in order to deduce that $F(\varphi_\delta)'\to F(\varphi)'$ in $L^2$.

In order to obtain that convergence, 
we observe that the above crude estimate can be refined by using the additional information that $u_1\in H^1$.
Since $|u_1'|=|F(\varphi)'|=|F'(\varphi)| \, |\varphi'|$,
that additional information implies indeed that, 
 away from the points where $|\sin\varphi|=|F'(\varphi)|=0$, we actually have $H^1$ control on $\varphi$. 
 
Observe that $B^{1/2}_{4,\infty}\subset C^{1/4}$ \cite[(2.7.1/12)]{triebel83}, that is  
 $[\varphi]_{1/4}\leq C \|\varphi\|_{B^{1/2}_{4,\infty}}$, where we denote by $[\varphi]_{1/4}$
 the $C^{1/4}$ seminorm of $\varphi$. 
We fix a small parameter $\lambda \geq 2 [\varphi]_{1/4}\delta^{1/4}$.
Then, from the definition of this seminorm, we have that
\begin{align}\label{eq:holder_varphi}
|\varphi(x)-\varphi(y)|\leq \frac{\lambda}{2\delta^{\frac 14}}
|x-y|^{\frac 14}\quad\forall x,y\in\R\,,
\end{align}
and we deduce in particular the property
\begin{equation}
\label{eq:lmb}
|x-y|\geq
16\,\delta
\quad
\text{ whenever }x,y
\text{ satisfy }
 |\sin{\varphi(y)}-\sin{\varphi(x)}|\geq\lambda
 \,.
\end{equation}
This property allows us to write $I$ as a union of mutually disjoint 
and not-too-small `elliptic' intervals
$\lbrace I_j^{e,\lambda}\rbrace$
 and `nonelliptic' intervals
$\lbrace I_k^{ne,\lambda}\rbrace$.
Specifically, we claim
\begin{align}
&
I
= \bigsqcup_j I_j^{e,\lambda} \sqcup \bigsqcup_k I_k^{ne,\lambda},
\qquad |I_j^{e,\lambda}|,|I_k^{ne,\lambda}|\geq 8\delta\,,
\nonumber
\\
&
|\sin\varphi|\geq \lambda\text{ in each }I^{e,\lambda}_j,
\quad
|\sin\varphi|\leq 2\lambda\text{ in each }I^{ne,\lambda}_k\,.\label{eq:ellipt_nonellipt_int}
\end{align}
On the intervals $I^{e,\lambda}_j$ the integrand
$|u_1'|^2=|\sin(\varphi)|^2 |\varphi'|^2$ is `elliptic' because 
$|\sin \varphi|$ stays away from $0$,
while on the intervals $I^{ne,\lambda}_k$ it is `nonelliptic' because $|\sin\varphi|$ stays close to zero.
Note that two disjoint elliptic intervals are separated by a non-elliptic interval of length at least $8\delta$,  so that the slightly enlarged elliptic intervals
$I^{e,\lambda}_j+(-2\delta,2\delta)$ 
(whose left/right endpoint are that of $I^{e,\lambda}_j$ shifted left/right by $2\delta$)
are still disjoint.
And the same holds for the slightly enlarged non-elliptic intervals
$I^{ne,\lambda}_k +(-2\delta,\delta)$.

In order to obtain this decomposition of $I$, 
one can for instance 
monitor the variations of the continuous function $f(x)=|\sin\varphi(x)|$ along the interval 
$\overline I=[a,b]$.
We set $x_0=a$ and define by induction a strictly increasing sequence  $(x_\ell)$ satisfying
\begin{align*}
x_{\ell +1} &
=\max \big( 
\sup J^{ne}_\ell, \sup J^e_\ell \big)\,,
\\
\text{where }
J^{ne}_\ell&=\lbrace x > x_\ell \colon f(x) \leq  2\lambda\text{ on }[x_\ell,x)\rbrace,
\\
J^e_\ell&
=\lbrace x > x_\ell \colon f(x) \geq \lambda\text{ on }[x_\ell,x)\rbrace\,.
\end{align*}
Here we adopt the convention that the supremum of an empty interval is  $-\infty$.
If one of these two intervals is empty, the other interval must have length at least $16\delta$ due to \eqref{eq:lmb}.
And if both intervals are non-empty (which actually happens only in the first iteration), then the largest of them 
must contain a transition of $f$ between $\lambda$ and $2\lambda$, which again has length at least $16\delta$ due to \eqref{eq:lmb}.
Therefore the sequence $(x_\ell)$ defined that way satisfies 
$x_{\ell +1}-x_\ell\geq 16\delta$.
Moreover, the definition of $x_{\ell +1}$ clearly implies
that the interval $I_\ell = (x_\ell,x_{\ell+1}]$ is either elliptic (i.e. $f\geq \lambda$ on $I_\ell$) or nonelliptic (i.e. $f\leq 2\lambda$ on $I_\ell$).
There exists an integer $N\geq 0$ such that $x_N < b \leq x_{N+1}$.
If $b-x_N\geq 8\delta$ we set $I_N=(x_n,b)$, and the disjoint intervals $I_0,\ldots, I_N$ are all of length at least $8\delta$ and either elliptic or nonelliptic.
If $b-x_N< 8\delta$, 
then thanks to \eqref{eq:lmb} and since $f(x_N)\in \lbrace \lambda,2\lambda\rbrace$, 
the interval $I_N=(b-8\delta,b)$ must be either elliptic or nonelliptic, 
and we modify $I_{N-1}$ into $I_{N-1}=(x_{N-1},b-8\delta]$, which is still either nonelliptic or elliptic, and has length at least $8\delta$.
In that case again, the disjoint intervals $I_0,\ldots, I_N$ 
are all of length at least $8\delta$ and either elliptic or nonelliptic. In all cases, we have obtained \eqref{eq:ellipt_nonellipt_int}. 

Now we come back to estimating 
$F(\varphi_\delta)'-[F(\varphi)_\delta]'$,
taking advantage of that decomposition.
On any elliptic interval
$I^{e,\lambda}_j$ we have $|\sin\varphi|\geq \lambda$,
recall \eqref{eq:ellipt_nonellipt_int}.
Thanks to the Hölder regularity 
\eqref{eq:holder_varphi}, we have in fact the similar lower bound
$|\sin\varphi|\geq c\lambda$ for $c=1-2^{-3/4}>0$,
on the 
enlarged interval 
$I^{e,\lambda}_j +(-2\delta,2\delta)$.
As a consequence, we obtain
\begin{align}\label{eq:ellipt_int}
\int_{I^{e,\lambda}_j +(-2\delta,2\delta)}(\varphi')^2\, dx 
\lesssim \frac{1}{\lambda^2}
\int_{I^{e,\lambda}_j + (-2\delta,2\delta) }(F(\varphi)')^2\, dx,
\end{align}
and for any $y,z\in [-1,1]$,
\begin{align*}
&
\int_{I^{e,\lambda}_j} (\varphi(x-\delta y)-\varphi(x-\delta z))^4\, dx
\\
&
\leq \sqrt{2} [\varphi]_{1/4}^{2}\delta^{1/2} \int_{I^{e,\lambda}_j}(\varphi(x-\delta y)-\varphi(x-\delta z))^2\, dx 
\\
&
=\sqrt{2} [\varphi]_{1/4}^{2}\delta^{1/2} \int_{I^{e,\lambda}_j}
\left( \int_0^1\varphi'(x-\delta y +t \delta(y-z))\, dt \delta (y-z)\right)^2\, dx
\\
&\leq 4\sqrt{2} [\varphi]_{1/4}^{2}\delta^{1/2+2}\int_{I^{e,\lambda}_j}
 \int_0^1\left(\varphi'(x-\delta y +t \delta(y-z))\right)^2\, dt \, dx
\\
&
\leq  4\sqrt{2}
[\varphi]_{1/4}^{2}\delta^{1/2+2}
\int_{I^{e,\lambda}_j +(-2\delta,2\delta)}(\varphi')^2\, dx.
\end{align*}
Then, invoking \eqref{eq:ellipt_int}, it follows that
\[
\int_{I^{e,\lambda}_j} 
(\varphi(x-\delta y)-\varphi(x-\delta z))^4\, dx \lesssim [\varphi]_{1/4}^{2}\frac{\delta^{1/2+2}}{\lambda^2}
\int_{I^{e,\lambda}_j +(-2\delta,2\delta)}(F(\varphi)')^2\, dx.
\]
Since the intervals $I_j^{e,\lambda}+(-2\delta,2\delta)$ are disjoint, on the union $I^{e,\lambda}=\bigcup_j I_j^{e,\lambda}$ we deduce, recalling also \eqref{eq:estimFvarphidelta},
\begin{align*}
\int_{I^{e,\lambda}}
\left (F(\varphi_\delta)'-[F(\varphi)_\delta]'\right)^2
\, dx 
&
\lesssim
 [\varphi]_{1/4}^{2}\frac{\delta^{1/2}}{\lambda^2}
\int_I(F(\varphi)')^2\, dx.
\end{align*}
On the nonelliptic intervals we have simply, recalling that $F^{-1}$ is $C^{1/2}$,
\begin{align*}
&
\int_{I^{ne,\lambda}_k}(\varphi(x-\delta y)-\varphi(x-\delta z))^4\, dx
\\
&
\lesssim 
\int_{I^{ne,\lambda}_k}(F(\varphi(x-\delta y))-F(\varphi(x-\delta z)))^2\, dx
\\
&
 \lesssim \delta^2
 \int_{I^{ne,\lambda}_k +(-2\delta,2\delta)}(F(\varphi)')^2\, dx,
\end{align*}
and on their union $I^{ne,\lambda}=\bigcup_k I_{k}^{ne,\lambda}$ we obtain, using again \eqref{eq:estimFvarphidelta},
\begin{align*}
\int_{I^{ne,\lambda}}
\left (F(\varphi_\delta)'-[F(\varphi)_\delta]'\right)^2
\, dx
& 
\lesssim \int_{I^{ne,\lambda}+(-2\delta,2\delta)} (F(\varphi)')^2\, dx 
\\
&
\leq \int_{\lbrace |\sin\varphi|\leq 3\lambda\rbrace }(F(\varphi)')^2\, dx.
\end{align*}
For the last inequality we used the fact that $|\sin(\varphi)|\leq 2\lambda$ on $I^{ne,\lambda}$ and the Hölder continuity \eqref{eq:holder_varphi} of $\varphi$.
Summing the estimates on $I^{e,\lambda}$ and $I^{ne,\lambda}$ we infer
\begin{align*}
\int_I
\left (F(\varphi_\delta)'-[F(\varphi)_\delta]'\right)^2
\, dx
&
\lesssim 
[\varphi]_{1/4}^{2}\frac{\delta^{1/2}}{\lambda^2}
\int_I
(F(\varphi)')^2\, dx
\\&+ 
\int_{\lbrace |\sin\varphi|\leq 3\lambda\rbrace }(F(\varphi)')^2\, dx,
\end{align*}
hence, for any $\lambda>0$,
\begin{align*}
\limsup_{\delta\to 0}
\int_I
\left (F(\varphi_\delta)'-[F(\varphi)_\delta]'\right)^2
\, dx
\lesssim 
\int_{\lbrace |\sin\varphi|\leq 3\lambda\rbrace }(F(\varphi)')^2\, dx.
\end{align*}
Letting $\lambda\to 0$, and recalling $u=e^{i\varphi}$, we deduce
\begin{align*}
\limsup_{\delta\to 0}
\int_I
\left (F(\varphi_\delta)'-[F(\varphi)_\delta]'\right)^2
\, dx
&
\lesssim 
\int_{\lbrace \sin\varphi =0\rbrace}(F(\varphi)')^2\, dx
\\
&
=\int_{\lbrace |u_1| =1\rbrace} (|u_1|')^2\, dx=0.
\end{align*}
The last equality follows e.g. from \cite[Lemma~3.60]{leoni}.
Since $|u_1'|=|F(\varphi)'|\in L^2$, we have
 $[F(\varphi)_\delta]'\to F(\varphi)'$ strongly in $L^2$, 
and we deduce that $F(\varphi_\delta)'$ converges to $F(\varphi)'$ strongly in $L^2$, 
which implies \eqref{eq:recov1Dudelta}. 
\end{proof}

%With the same calculations as in the 2D case (see also calculations below), we have that 
%\begin{align*}
%\Glim E_\e (u)=\int_{-1}^1 (u_1')^2\, dx \qquad\text{ for }u\in B^{1/2}_{4,c_0}(-1,1;\mathbb S^1)\text{ s.t. }u_1\in H^1(-1,1).
%\end{align*}
%Moreover, in 1D we have $B^{1/2}_{4,\infty}\subset VMO$, so no need for the equiintegrability condition and we have
%\begin{align*}
%\Glimsup E_\e (u) <\infty 
% \qquad\text{ for }u\in B^{1/2}_{4,\infty}(-1,1;\mathbb S^1)\text{ s.t. }u_1\in H^1(-1,1).
%\end{align*}
%In this 1D setting, we can prove the converse estimate:
%\begin{align*}
%\Glimsup E_\e(u)<\infty \quad \Rightarrow
%\quad
%u\in B^{1/2}_{4,\infty}(-1,1;\mathbb S^1).
%\end{align*}
%This is a consequence of the following bound:
%\begin{align}\label{e:reg1D}
%|u|_{B^{1/2}_{4,\infty}}^4
%\lesssim 
%\int_{-1}^1 (u_1')^2\, dx
%\qquad
%\forall u\in H^1(-1,1;\mathbb S^1).
%\end{align}
%To prove \eqref{e:reg1D}, write $u=e^{i\varphi}$ for some $\varphi\in H^1(-1,1;\R)$. 
%
%
%
%
%\medskip
%
%To summarize: we have
%\begin{align*}
%X :=\left\lbrace \Glimsup E_\e
%<\infty\right\rbrace
%= 
%\left\lbrace
%u\in B^{1/2}_{4,\infty}(-1,1;\mathbb S^1)\text{ s.t. }u_1\in H^1(-1,1)
%\right\rbrace,
%\end{align*}
%and for $u\in X\cap B^{1/2}_{4,c_0}$,
%we know in addition that the $\Gamma$-limit is given by $E_0(u)=\int_{-1}^1(u_1')^2\, dx$.
%
%
%\medskip
%
%It would remain to determine the $\Gamma$-limit on $X\setminus B^{1/2}_{4,c_0}$. 
%
%
%
%
%\medskip 

\section{A thin film limit}\label{s:thin}

In this section we prove Theorem~\ref{t:thin} about the $\Gamma$-convergence of the energy
\begin{align*}
E_{\e,\delta}(u)=
\frac 12
\int_{[-1,1]^2} 
\left(\partial_1u_{1} + \frac 1\delta \partial_2 u_2\right)^2 
+
\e \left( \partial_1 u_2 -\frac 1\delta \partial_2 u_1\right)^2\, dx,
\end{align*}
for $u\in H^1([-1,1]^2;\mathbb S^1)$
subject to the boundary conditions $u(\pm 1,x_2)=e_\pm$ for  $x_2\in (-1,1)$, and $u(x_1,-1)=u(x_1,+1)$.

We start by adapting the compensated regularity estimate of Theorem~\ref{t:reg_estim} to this thin periodic case,
thereby proving the equiboundedness and compactness statement in Theorem~\ref{t:thin}.

\begin{lemma}\label{l:thin_reg}
For any  $u\in H^1([-1,1]^2;\mathbb S^1)$
satisfying
 $u(\pm 1,x_2)=e_\pm$ for  $x_2\in (-1,1)$, and $u(x_1,-1)=u(x_1,+1)$ for  $x_1\in (-1,1)$, we have
\begin{align}\label{eq:thinbesovx1}
\int_{[-1,1]^2}|D^{te_1}u|^3 \, dx 
&
\lesssim |t|^{\frac 32}\left(1+
E_{\e,\delta}(u)^{\frac 32}\right)
\,,
\\
\label{eq:thinbesovx2}
\int_{[-1,1]^2}|D^{t e_2}u|^3\, dx
&
\lesssim 
\delta^{\frac 32}|t|^{\frac 32} \Big( 1 +E_{\e,\delta}(u)^{\frac 32} \Big)\,,
\end{align}
for all $t\in \R$.
Here, $u$ is extended to $\R^2$ by setting 
\[
u(x_1,x_2)=\begin{cases} e_+,&\mbox{ if }x_1\geq 1,\\
e_-,&\mbox{ if }x_1\leq -1\,,
\end{cases}
\] and $u(x_1,x_2 +2k)=u(x_1,x_2)$ for all $k\in\Z$.
\end{lemma}

\begin{proof}[Proof of Lemma~\ref{l:thin_reg}]
Define $\tilde u(x_1,x_2)=u(x_1,x_2/\delta)$ 
for $(x_1,x_2)\in \R^2$, so that
$\tilde u\in H^1_{\loc}$ and
\begin{align*}
\int_{[-1,1]\times [-\delta,\delta]} (\nabla\cdot \tilde u)^2\, dx \leq 2 \delta \, E_{\e,\delta}(u).
\end{align*}
Note that, from the definition of the extension of $u$ to $\R^2$, we have  $\tilde u(x_1,x_2+2\delta)=\tilde u(x_1,x_2)$ and $\tilde u(x_1,x_2)=\tilde u(\pm 1,x_2)$ for $x_1>1$ or $x_1<-1$.
Thus we find
\begin{align*}
\int_{[-2,2]^2} (\nabla\cdot \tilde u)^2\, dx 
&
\lesssim\frac 1\delta \int_{[-1,1]\times [-\delta,\delta]} (\nabla\cdot \tilde u)^2\, dx 
%\\
%&
\leq  2 E_{\e,\delta}(u),
\end{align*}
and, for any $p\in [1,\infty)$,
\begin{align*}
\int_{[-1,1]^2}|D^{te_1}u|^p \, dx 
&
=
\frac 1\delta \int_{[-1,1]\times [-\delta,\delta]}|D^{te_1}\tilde u|^p \, dx
%\\
%&
\lesssim 
 \int_{[-1,1]^2}|D^{te_1}\tilde u|^p \, dx\,.
\end{align*}
Theorem~\ref{t:reg_estim} implies, for $h\in\R^2$,
\begin{align}\label{eq:thinbesovtildeu}
\int_{[-1,1]^2}|D^{h} \tilde u|^3\, dx  
&
\lesssim |h|^{\frac 32} \left( \int_{[-2,2]^2} (\nabla\cdot \tilde u)^2\, dx +1 \right)^{\frac 32}
\nonumber
\\
&
\lesssim |h|^{\frac 32} \Big( 1 +E_{\e,\delta}(u)^{\frac 32} \Big)\,.
\end{align} 
Applying this to increments $h=te_1$ in the $x_1$-direction, we obtain \eqref{eq:thinbesovx1}.
For increments in the $x_2$-direction, we have
\begin{align*}
D^{ te_2}u(x)
&
=u(x_1,x_2+ t)-u(x_1,x_2)
=\tilde u(x_1, \delta x_2 + \delta t)-\tilde u(x_1,\delta x_2)
\\
&
=D^{\delta t e_2}\tilde u(x_1,\delta x_2),
\end{align*}
hence
\begin{align*}
\int_{[-1,1]^2}|D^{te_2}u|^p\, dx
=\frac 1 \delta \int_{[-1,1]\times [-\delta,\delta]}
|D^{\delta te_2}\tilde u|^p\, dx
\lesssim \int_{[-1,1]^2} |D^{\delta te_2}\tilde u|^p\, dx\,,
\end{align*}
so \eqref{eq:thinbesovtildeu} applied to $h=\delta te_2$ implies \eqref{eq:thinbesovx2}.
\end{proof}

Lemma~\ref{l:thin_reg} 
will enable us to show that limits of bounded energy sequences for $E_{\e,\delta}$ are one-dimensional and belong to $B^{1/2}_{3,\infty}([-1,1]^2;\mathbb S^1)$.
Next we prove a lemma 
which will allow us to conclude
 that such limits actually belong to the smaller space
$B^{1/2}_{4,\infty}([-1,1];\mathbb S^1)$.

\begin{lemma}\label{l:1Dcont}
Assume $u \colon [-1,1]\to\mathbb S^1$ is continuous and $u_1\in H^1([-1,1])$.
Then $u \in B^{1/2}_{4,\infty}([-1,1];\mathbb S^1)$.
\end{lemma}

%\noindent
%Is Lemma~\ref{l:1Dcont} still true if we assume only $ u$ is VMO ?

\begin{proof}[Proof of Lemma~\ref{l:1Dcont}]
Note that $|u_2|=\sqrt{1-u_1^2}$ 
is a composition of the $H^1$ function $u_1$ and the $C^{1/2}$ function $x\mapsto\sqrt{1-x^2}$.
Hence, by the
same argument as in Proposition~\ref{p:reg1D},
we have $|u_2| \in B^{1/2}_{4,\infty}([-1,1])$.
But this does not directly imply that 
$u_2\in B^{1/2}_{4,\infty}([-1,1])$,
and we use a slightly different argument to prove Lemma~\ref{l:1Dcont}.

Since $u$ is continuous,
one can define a continuous phase $\varphi$ such that $ u=e^{i\varphi}$ on $[-1,1]$.
Recall the function $F(t)=\int_0^t |\sin(s)|\,ds$ defined in \eqref{eq:F}.
If $\varphi$ were differentiable, 
we would have $|F(\varphi)'|=|\sin\varphi| |\varphi'| =|u_1'|\in L^2([-1,1])$, and conclude that $\varphi \in B^{1/2}_{4,\infty}([-1,1])$ since $F^{-1}$ is $C^{1/2}$, as in Proposition~\ref{p:reg1D}.

This calculation using the chain rule is only formal, because we do not know that $\varphi$ is (even weakly) differentiable.
But we do show in what follows that $F(\varphi)\in H^1([-1,1])$.
First we establish that
\begin{align}
&
\forall x\in [-1,1],\;\exists r_x>0
\text{ such that, for all }
y_1< y_2\in (x-r_x,x+r_x)\,,
\nonumber
\\
&\text{ we have }
\quad
|F(\varphi(y_1))-F(\varphi(y_2))|\leq 
\int_{y_1}^{y_2} |u_1'(y)|\, dy\,.
%|u_1(y_1)-u_1(x)|
%+|u_1(x)-u_1(y_2)|
%\,.
\label{eq:rx}
\end{align}
To prove  \eqref{eq:rx} we distinguish two cases
depending on the value of $\sin\varphi(x)$.

Let us first assume $\sin\varphi(x)\neq 0$.
There exists $r_x>0$ such that,
on the interval $(x-r_x,x+r_x)$,
the continuous function
$\varphi$ takes values in an interval where $\sin$ does not vanish. In such an interval, we have $F(\varphi(y))=C + \alpha \cos \varphi(y)$ for some constants $C\in\Z$, $\alpha\in \lbrace\pm 1\rbrace$, so $F(\varphi)=C + \alpha  u_1$. 
We deduce that $F(\varphi)\in H^1(x-r_x,x+r_x)$ and $|F(\varphi)'|=| u_1'|$ in $(x-r_x,x+r_x)$,
which implies \eqref{eq:rx}.

Let us now consider the remaining case $\sin\varphi(x)=0$.
There exists $r_x>0$ such that,
on the interval $(x-r_x,x+r_x)$,
 the continuous function $\varphi$ takes values in an interval of length less than $\pi$.
In that interval we have
\begin{align*}
F(\varphi(y))=\begin{cases}
C - \alpha\cos \varphi(y)\quad &\text{if }\varphi(y)>\varphi(x),
\\
C -2  + \alpha\cos \varphi(y)\quad & \text{if }\varphi(y)<\varphi(x),
\end{cases}
\end{align*}
for some constant $C\in\Z$, $\alpha=\cos\varphi(x)\in\lbrace \pm 1\rbrace$.
Let $y_1<y_2\in (x-r_x,x+r_x)$.
If $\varphi(y_1)<\varphi(x)<\varphi(y_2)$,
then there exists $\bar x\in (y_1,y_2)$ such that 
$\varphi(\bar x)=\varphi(x)$,
and  we deduce
\begin{align*}
 F( \varphi(y_2))-F(\varphi(y_1))
 &=-\alpha\cos \varphi(y_2)+2-\alpha\cos \varphi(y_1)
 \\
 &=   -\alpha\cos \varphi(y_2)+2\alpha\cos\varphi(\bar x)-\alpha\cos \varphi(y_1)
 \\
 &=\alpha\left(u_1(\bar x)-u_1(y_2)+u_1(\bar x)-u_1(y_1)\right).
\end{align*}
Therefore,
\begin{align*}
|F( \varphi(y_1))-F(\varphi(y_2))|
&
\leq |u_1(y_1)-u_1(\bar x)| + |u_1(y_2)-u_1(\bar x)|
\\
&
\leq \int_{y_1}^{\bar x} |u_1'|\, dy +\int_{\bar x}^{y_2}|u_1'|\, dy 
\leq \int_{y_1}^{y_2} |u_1'|\, dy\,.
\end{align*}
If $\varphi(y_1)$ and $\varphi(y_2)$ both lie on the same side of $\varphi(x)$,
the inequality \eqref{eq:rx}
 also holds because in that case 
we  have $|F(\varphi(y_2))-F(\varphi(y_1))|=|\cos \varphi(y_2)-\cos \varphi(y_1) |
=|u_1(y_2)-u_1(y_1)|$.
In all cases, we have proved \eqref{eq:rx}.

Applying  \eqref{eq:rx} 
to any 
$y_1<\cdots < y_M \in (x-r_x,x+r_x)$, we see that
\begin{align*}
\sum_{i=1}^{M-1}|F(\varphi(y_{i+1}))-F(\varphi(y_i))|
\leq \int_{y_1}^{y_M}|u_1'|\, dy\,,
\end{align*}
hence the function $F(\varphi)$ is absolutely continuous on $(x-r_x,x+r_x)$,
so it is differentiable almost everywhere.
Moreover, the inequality \eqref{eq:rx} implies that its derivative
satisfies $|F(\varphi)'|\leq |u_1'|$
 almost everywhere on $(x-r_x,x+r_x)$.
 
 This is true for any $x\in [-1,1]$,
 %and the open intervals  $(x-r_x,x+r_x)$ cover $[-1,1]$,
 so we deduce
 that $F(\varphi)$ is absolutely continuous on $[-1,1]$ and
 its derivative satisfies $|F(\varphi)'|\leq |u_1'|$.
 Recalling that $u_1\in H^1([-1,1])$, this implies $F(\varphi)\in H^1([-1,1])$.
We deduce as in Proposition~\ref{p:reg1D}
 that $\varphi\in B^{1/2}_{4,\infty}(-1,1)$,
 and therefore $ u\in B^{1/2}_{4,\infty}([-1,1];\mathbb S^1)$ 
since $u=e^{i\varphi}$ is a Lipschitz function of $\varphi$.
\end{proof}

\begin{proof}[Proof of Theorem~\ref{t:thin}]
Let
 $(u^{(j)})\subset H^1([-1,1]^2;\mathbb S^1)$  
 such that
$E_{\e_j,\delta_j}(u^{(j)})\leq C$ for some $(\e_j,\delta_j)\to (0,0)$.
In light of Jensen's inequality and the $x_2$-periodicity one has
\begin{align*}
E_{\e_j,\delta_j}(u^{(j)})
&
\geq
 \int_{-1}^1 
 \bigg[
 \bigg( \frac{d\bar u^{(j)}_{ 1}}{dx_1} \bigg)^2 
 +\e  \bigg(  \frac{d\bar u^{(j)}_{ 2}}{dx_1}\bigg)^2
 \bigg]\, dx_1,
 \\
 \bar u^{(j)}(x_1)
 &
 := \frac 12\int_{-1}^1 u^{(j)}(x_1,x_2)\, dx_2.
\end{align*}
Thanks to \eqref{eq:thinbesovx1}, 
the averaged maps $\bar u^{(j)}$ are equibounded in 
$B^{1/2}_{3,\infty}([-1,1];\R^2)$.
We can extract a subsequence such that 
$\bar u^{(j)} \to \bar u$ strongly in $L^2$,
and $d\bar u^{(j)}_1/dx_1$ converges weakly in $L^2$.
Then we have $\bar u_1'\in L^2$ and
\begin{align*}
%\label{eq:thinliminfbaru}
\liminf_{\e\to 0}
E_{\e,\delta}(u_\e) \geq \int_{-1}^1(\bar u_1')^2\, dx.
\end{align*}
Moreover,
Lemma~\ref{l:thin_reg}
 implies that $(u^{(j)})$
is bounded in $B^{1/2}_{3,\infty}([-1,1]^2;\mathbb S^1)$, 
and any weak limit $u$ in that space satisfies
$D^{te_2}u=0$ for all $t\in\R$.
Hence we have
$u=u(x_1)=\bar u(x_1)\in B^{1/2}_{3,\infty}([-1,1];\mathbb S^1)\subset C([-1,1];\mathbb S^1)$.
Further, applying Lemma~\ref{l:1Dcont}  we deduce that $u=\bar u\in B^{1/2}_{4,\infty}([-1,1];\mathbb S^1)$.
This proves
the lower bound part of the $\Gamma$-convergence statement of Theorem~\ref{t:thin},
as well as the equiboundedness and compactness statement.

For the upper bound, since bounded energy limits are one-dimensional, we can directly use the construction from \textsection\ref{s:1D},
concluding the proof of Theorem~\ref{t:thin}.
 \end{proof}

\begin{appendices}

\section{Traces}\label{a:traces}

In this Appendix we sketch the proof of Proposition~\ref{p:tracelim},
adapting directly the arguments in \cite{vasseur01} to our slightly different context.

%\begin{proposition}\label{p:tracelim}
%If $u\in B^{1/2}_{3,\infty,\loc}(\Omega;\mathbb S^1)$ satisfies
%$\dv u\in L^2(\Omega)$ and $\dv\Phi(u)=\lambda_\Phi(u)\dv u$ for all $\Phi\in\mathrm{ENT}$, 
%then $u$ admits a strong trace on $\partial\Omega$, that is,
%there exists a measurable map $\tr(u)\colon\partial\Omega\to\mathbb S^1$ such that
%\begin{align*}
%\esslim_{\delta\to 0}\int_{\partial\Omega} |u(x-\delta\nu(x))-\tr(u)(x)|\, d\mathcal H^1(x)= 0\,,
%\end{align*}
%where $\nu(x)$ denote the exterior unit normal to $x\in\partial\Omega$.
%\end{proposition} 

\begin{proof}[Proof of Proposition~\ref{p:tracelim}]
There are two main steps: 
first, the existence of a weak trace for the indicator function $\chi$, 
and, second, the proof that this weak trace is an indicator function.
This last fact then implies that the trace is strong.

\medskip
\noindent{\it Step 1: Existence of a weak trace.} There exists $\tr(\chi)\in L^\infty(\T\times\partial\Omega)$ 
a weak-$*$ trace of the indicator function $\chi(s,x)=\mathbf 1_{e^{is}\cdot u(x)>0}$, in the following sense.
For any $x_0\in\partial\Omega$ and any $C^1$-diffeomorphism
\begin{align}\label{eq:psi_diffeo}
&\psi \colon V\to U,\quad 0\in V,\; x_0\in U\,,
\\
&\psi^{-1}(U\cap\Omega)=V\cap H_+,\quad H_+=\lbrace x\in\R^2\colon x_2>0\rbrace\,,
\nonumber
\\
&
\overline U\cap\partial\Omega =\overline{U\cap \partial\Omega}\,,
\nonumber
\\
&
\psi^{-1}( U\cap\partial\Omega)= V\cap\partial H_+=V\cap\lbrace x_2=0\rbrace\,,
\nonumber
\end{align}
we have
\begin{align*}
\chi\circ\big(\mathrm{id}_\T\otimes \psi(\cdot,\delta) \big) 
\overset{*}{\rightharpoonup} \tr(\chi)\circ
\big(\mathrm{id}_{\T}\otimes \psi(\cdot,0)\big)\,,
%\quad
%\text{weakly-* in }L^\infty(\T\times (V\cap \partial H_+))\,,
\end{align*}
weakly-$*$ in $L^\infty(\T\times (V\cap\partial H_+))$
as $\delta\to 0$.
%for all $\xi\in L^1(\T\times V\cap\partial H_+)$.
\medskip

\noindent
{\it Proof of Step 1.}
Let $\omega_\delta=\psi(V\cap\lbrace x_2>\delta\rbrace)$, $\gamma_\delta=\psi(V\cap\lbrace x_2=\delta\rbrace)$.
First note that
 $u\in B^{1/2}_{3,\infty}(\omega_\delta)$,
hence classical trace properties (see e.g. \cite[\S~3.3.3]{triebel83})
imply that $u$ has a strong trace on $\gamma_\delta$,
and 
therefore $\tilde \chi_\delta:= \chi\circ (\mathrm{id}_\T\otimes \psi(\cdot,\delta)) $ 
can be interpreted as the strong trace
\begin{align*}
\tilde \chi_\delta(s,x_1)=\mathbf 1_{e^{is}\cdot u(\psi(x_1,\delta))>0}\,.
\end{align*}
We also define 
$\chi_\delta\colon\T\times (U\cap\partial\Omega)\to\lbrace 0,1\rbrace$ by setting
\begin{align*}
\chi_\delta(s,\psi(x_1,0))=\tilde\chi_\delta(s,x_1)\qquad\forall (s,x_1)\in\T\times (V\cap\partial H_+)\,.
\end{align*}
This bounded sequence in $L^\infty(\T\times U\cap\partial\Omega)$ admits a weak-$*$ converging subsequence
\begin{align*}
\chi_{\delta_j}\overset{*}{\rightharpoonup }
\chi_0\in L^\infty(\T\times (U\cap \partial \Omega))\,.
\end{align*}
Next
 we show that this weak-$*$ limit $\chi_0$ does not depend on the subsequence $\delta_j\to 0$ nor the choice of the diffeomorphism $\psi$. 
That is, different choices of $\psi$ give the same limit $\chi_0$ almost everywhere on the intersection of their domains.
As a consequence, we can set $\tr(\chi)=\chi_0$ and it is a well-defined weak-$*$ trace, thus proving Step~1.

To check this, note that
thanks to the strong trace property on $\gamma_\delta$,
  we have the Gauss-Green formula
\begin{align*}
\int_{\omega_\delta} \eta \, \dv \Phi(u) \, dx
+\int_{\omega_\delta} \langle \nabla\eta,\Phi(u) \rangle\, dx 
=\int_{\gamma_\delta} \eta\, \langle \nu_\delta,\Phi(\tr_{\gamma_\delta}(u))\rangle\, d\mathcal H^1\,,
\end{align*} 
for all $\eta\in C_c^1(U)$ and $\Phi\in\mathrm{ENT}$,
where $\nu_\delta$ is the exterior unit normal along $\gamma_\delta$, see e.g. \cite{CF99}.
Since $\Phi(u)\in L^\infty(\Omega)$, $\dv\Phi(u)\in L^2(\Omega)$,
and $|U\cap\Omega\setminus \omega_\delta|\to 0$,  we deduce
that the function
\begin{align*}
f_{\eta,\Phi}^\psi(\delta)=
\int_{\gamma_\delta}
\eta\, \langle \nu_\delta,\Phi(\tr_{\gamma_\delta}(u))\rangle\, d\mathcal H^1\,,
\end{align*}
is Cauchy as $\delta\to 0$.
Moreover, if $\psi,\psi'$ are two different diffeomorphisms
whose image contains the support of $\eta$, 
then the Gauss-Green formula applied in the corresponding sets $\omega_\delta$, $\omega'_\delta$ also implies that
\begin{align*}
f_{\eta,\Phi}^\psi(\delta)-f_{\eta,\Psi}^{\psi'}(\delta)\to 0
\quad\text{as }\delta\to 0\,.
\end{align*}
We deduce that there exists $\ell_{\eta,\Phi}\in\R$ such that
\begin{align*}
\int_{\gamma_\delta}\eta\, \langle \nu_\delta,\Phi(\tr_{\gamma_\delta}(u))\rangle\, d\mathcal H^1 \to \ell_{\eta,\Phi}\quad\text{as }\delta\to 0\,,
\end{align*}
for any choice of diffeomorphism $\psi$ whose image contains the support of $\eta$.
Applying this to the entropies $\Phi_g$ given in \eqref{eq:Phi_g},
and using that 
$\psi(\cdot,\delta)\circ\psi(\cdot,0)^{-1}\to\mathrm{id}_{U\cap\partial\Omega}$ in $C^1_{\loc}(U\cap \partial\Omega)$, we infer
\begin{align*}
\int_{\T\times(U\cap\partial\Omega)}
g(s)\eta(x)\chi_\delta(s,x)
\langle e^{is},\nu(x)\rangle\, ds d\mathcal H^1(x)\to\ell_{\eta,\Phi_g}\quad\text{as }\delta\to 0\,,
\end{align*}
for all $\eta\in C_c^1(\R^2)$ and $g\in C^0(\T)$,
%where $\bar\nu_\delta =\nu_\delta\circ\psi(\cdot,\delta)\circ\psi(\cdot,0)^{-1}$.
%
Recalling the weak star convergence $\chi_{\delta_j}\overset{*}{\rightharpoonup}\chi_0$, this implies that the quantity
\begin{align*}
\int_{\T\times(U\cap \partial\Omega)}g(s)\eta(x)\chi_0(s,x)
\langle e^{is},\nu(x)\rangle\, ds d\mathcal H^1(x)  = \ell_{\eta,\Phi_g}\,,
\end{align*}
is independent of the sequence $\delta_j\to 0$ and of the choice of  $\psi$.
This being true for all $\eta$ and $g$, and since 
$\langle e^{is},\nu(x)\rangle \neq 0$ a.e. in $\T\times\partial\Omega$, 
we deduce that
$\chi_0(s,x)$ is independent of the sequence $\delta_j\to 0$ and that two different choices of diffeomorphisms $\psi$ give the same limit $\chi_0$ on the intersection of their domains.

\medskip
\noindent
{\it Step 2. The weak trace is an indicator function.} 
We have $\chi_0=\tr(\chi)\in\lbrace 0,1\rbrace$ a.e. on 
$\T\times\partial\Omega$.
\medskip

\noindent
{\it Proof that Step 2 implies the conclusion of Proposition~\ref{p:tracelim}.}
With the notations of Step~1, if $\chi_0$ takes values in $\lbrace 0,1\rbrace$, 
then we have $(\chi_0)^2=\chi_0$ and therefore
\begin{align*}
\int_{\T\times (U\cap\partial\Omega)}(\chi_\delta)^2 \, dsd\mathcal H^1
&
=\int_{\T\times (U\cap \partial\Omega)}\chi_\delta \, dsd\mathcal H^1
\\
&
\to
\int_{\T\times(U\cap\partial\Omega)}\chi_0 \, dsd\mathcal H^1 
=\int_{\T\times(U\cap \partial\Omega)}(\chi_0)^2 \, dsd\mathcal H^1\,,
\end{align*}
so that $\chi_\delta\to \chi_0$ strongly in $L^2(\T\times(U\cap\partial\Omega))$,
and this in turn implies that
\begin{align*}
\tr_{\gamma_\delta}(u)(\psi(x_1,\delta))=\frac 12 \int_{\T}\chi_\delta(s,\psi(x_1,0))e^{is}\, ds\,,
\end{align*}
converges strongly in $L^2(V\cap\partial H_+)$.
Covering $\partial\Omega$ 
with coordinate patches and choosing adequate diffeomorphisms $\psi$ concludes the proof of Proposition~\ref{p:tracelim}.
\qed

\medskip
\noindent
{\it Proof of Step~2.}
This is obtained via a blow-up argument in the kinetic equation \eqref{eq:kin}.
The main idea is that, 
when blowing up at a generic boundary point $x_0\in\partial\Omega$, 
%the kinetic equation provides enough compactness to ensure that 
the blow-up limit is an indicator function
which
%moreover 
satisfies the free transport equation $e^{is}\cdot\nabla\chi=0$ and has therefore a strong boundary trace.
That strong trace of the blow-up limit can then be shown to coincide with the blow-up limit of the weak trace $\chi_0$, which must therefore take values in $\lbrace 0,1\rbrace$.

To perform the blow-up argument, we fix $x_0\in\partial\Omega$,
%choose coordinates in which $\nu(x_0)=-\mathbf e_2$,
%and fix 
and a diffeomorphism $\psi\colon V\to U$ as in \eqref{eq:psi_diffeo}.
Then we define $\tilde\chi\colon\T\times( V\cap H_+)\to\lbrace 0,1\rbrace$ by setting
\begin{align*}
\tilde\chi(s,x)=\chi(s,\psi(x))=\mathbf 1_{e^{is}\cdot \tilde u(x)>0}\,,\quad \tilde u(x)=u(\psi(x))\,,
\end{align*}
and from the kinetic equation \eqref{eq:kin} satisfied by $\chi$ we deduce 
\begin{align*}
[D\psi(x)^{-1}e^{is}]\cdot\nabla_x\tilde\chi =\widetilde \Theta \in L^2(V\cap H_+;\mathcal M(\T))\,.
\end{align*}
For any $x_*\in V\cap \partial H_+$ and $\e>0$ we define
\begin{align*}
\hat\chi_\e(s,\hat x)&
=\tilde\chi(s,x_* +\e\hat x) =\mathbf 1_{e^{is}\cdot \hat u_\e(\hat x)>0}\,,
\quad \hat u_\e(\hat x)=\tilde u(x_*+\e\hat x)\,,
\\
\widehat\Theta(s,\hat x)
&
=\frac{1}{\e}\widetilde\Theta(s,x_*+\e\hat x)\,,
\end{align*}
which satisfy
\begin{align*}
[D\psi(x_*+\e\hat x)^{-1}e^{is}]\cdot\nabla\hat\chi_\e
=\widehat \Theta_\e \,,
\end{align*}
in $\mathcal D'(\T\times (B_2\cap H_+))$ for small $\e>0$.
In order to simplify the expression of $D\psi(x)^{-1}$ we choose coordinates on $U$ in which $\nu(x_0)=-\mathbf e_2$,
 and fix a diffeomorphism $\psi$ in graphical form
\begin{align*}
\psi(x_1,x_2)=(x_1,h(x_1)+x_2)\,,
\end{align*}
for some $h\in C^1(V\cap \lbrace x_2=0\rbrace)$. That way, we have
\begin{align*}
D\psi(x)^{-1}v = (v_1,v_2-h'(x_1)v_1)\quad\forall v\in\R^2\,,
\end{align*}
and can therefore rewrite the above kinetic equation for $\hat\chi_\e$ as
\begin{align}\label{eq:hatchieps_kin}
[D\psi(x_*)^{-1}e^{is}]\cdot\nabla\hat\chi_\e
&
=(\cos(s)\partial_{\hat x_1} +(\sin(s)-h'(x_{*1})\cos(s))\partial_{\hat x_2})\hat\chi_\e
\nonumber
\\
&
=
\widehat \Theta_\e + \partial_{\hat x_2}\hat g_\e\,,
\\
\hat g_\e 
&
=
[ (h'(x_{*1})-h'(x_{*1}+\e\hat x_1))\cos(s) \hat\chi_\e]\,.
\nonumber
\end{align}
The terms in the right-hand side of that kinetic equation satisfy
$\hat g_\e\to 0$ uniformly on $\T\times (B_2\cap H_+)$, and
\begin{align*}
\int_{\T\times (B_2\cap H_+)}|\widehat\Theta_\e(s,\hat x)|\, ds d\hat x
&
=\frac{1}{\e}\int_{\T\times (B_{2\e}(x_*)\cap H_+)}|\widetilde\Theta(s,x)|\, ds dx
\\
&
\lesssim \|\widetilde \Theta\|_{L^2(B_{2\e}(x_*)\cap H_+;\mathcal M(\T))}\to 0\,,
\end{align*}
as $\e\to 0$.
Arguing as in \cite[Proposition~3]{vasseur01}
 relying e.g. on the velocity averaging results of \cite{PS98}, we deduce that
we can extract a (non-relabeled) sequence $\e\to 0$
such that 
\begin{align*}
\hat u_\e(\hat x)=\frac 12 \int_\T \hat\chi_\e(s,\hat x)e^{is}\, ds
\to \hat u_*(\hat x)
\quad\text{in }L^1_{\loc}(B_2\cap H_+)\,,
\end{align*}
for some $\mathbb S^1$-valued map $\hat u_*$.
Letting $\hat\chi_*(s,\hat x)=\mathbf 1_{e^{is}\cdot \hat u_*(\hat x)>0}$, and since 
\begin{align*}
\int_{\T}|\hat \chi_\e(s,\hat x)-\hat \chi_*(s,\hat x)|\, ds \lesssim |\hat u_\e(\hat x)-\hat u_*(\hat x)|\,,
\end{align*}
we deduce also that
\begin{align*}
&
\hat\chi_\e\to\hat\chi_*\quad\text{ in }L^1_{\loc}(\T\times(B_2\cap H_+))\,.
\end{align*}
Moreover, if we consider $\hat\chi_\e$ and $\hat\chi_*$ as extended by 0 in $B_1\setminus H_+$, then we have $\hat\chi_\e\to\hat\chi_*$ in $L^1_{\loc}(\T\times K)$ for any compact set $K\subset B_2\setminus \lbrace x_2=0\rbrace$.
Since $(\hat\chi_\e)$ is weakly compact in $L^2_{\loc}(B_2)$,
we also have $\hat\chi_\e\to\hat\chi_*$ weakly in $L^2_{\loc}(B_2)$.
But we have already seen that this, combined with the fact that $\hat\chi_\e$ and $\hat\chi_*$ take values into $\lbrace 0,1\rbrace$, implies strong convergence. 
Thus we have
\begin{align*}
\mathbf 1_{H_+}\hat\chi_\e \to \mathbf 1_{H_+}\hat\chi_*\quad\text{ in }L^1_{\loc}(\T\times B_2)\,.
\end{align*}
Further, recall that the right-hand side of the kinetic equation \eqref{eq:hatchieps_kin} satisfied by $\hat\chi_\e$ converges to $0$ in distributions, so that its limit $\hat\chi_*$ satisfies the free transport equation
\begin{align*}
[D\psi(x_*)^{-1}e^{is}]\cdot \nabla_{\hat x}\hat\chi_* =0\quad\text{in }\mathcal D'(B_2\cap H_+)\,.
\end{align*}
This implies that, for a.e. $s\in\T$, the function $\hat\chi_*(s,\cdot)$ 
depends only on the variable $\hat x\cdot (iD\psi(x_*)^{-1}e^{is})$ orthogonal to the direction
$D\psi(x_*)^{-1}e^{is}$.
For a.e. $s\in\T$, that variable is transverse to $\lbrace x_2=0\rbrace$, 
so that $\hat \chi_*(s,\cdot)$ has a strong $L^1_{\loc}$ trace onto $B_2\cap \lbrace x_2=0\rbrace$.
By dominated convergence,
we deduce that
 $\hat\chi_*$ 
has a local strong $\lbrace 0,1\rbrace$-valued trace $\tr(\hat\chi_*)$
on $\T\times (B_2\cap \lbrace \hat x_2=0\rbrace)$,
and in particular
\begin{align*}
\esslim_{\delta\to 0}\int_{\T\times [-1,1]}
\!\!\!
|\hat\chi_*(s,(\hat x_1,\delta))-\tr(\hat\chi_*)(s,(\hat x_1,0))|\, dsd\hat x_1 =0
\end{align*}
Finally we make the link between this trace of $\hat\chi_*$ and the weak trace $\chi_0$, 
by  considering, for any $\zeta\in C_c^1(\T\times (-1,1))$, the function 
\begin{align*}
\rho_\e(\hat x_2)
&
=\int_{\T\times [-1,1]}
\!\!\!
(\hat a_\e(s,\hat x)\hat\chi_\e(s,\hat x)-
\hat a_*(s)\hat\chi_*(s,\hat x))\zeta(s,\hat x_1)\, ds d\hat x_1\,,
\\
\hat a_\e(s,\hat x)
&
=\sin(s)-h(x_{*1}+\e\hat x_1)\cos(s),\quad
\hat a_*(s)=\sin(s)-h(x_{*1})\cos(s)\,,
\end{align*}
defined for $0<\hat x_2\leq 1$.
The weak trace property established in Step~1 ensures
\begin{align*}
&
\rho_\e(0^+)=
\lim_{\hat x_2\to 0^+}\rho_\e(\hat x_2)
\\
&
= \int_{\T\times [-1,1]}
\!\!\!
(\hat a_\e(s,\hat x)\tilde\chi_0(s,x_* +\e \hat x_1)-
\hat a_*(s)\hat\chi_*(s,\hat x_1))\zeta(s,\hat x_1)\, ds d\hat x_1\,,
\end{align*}
 where $\tilde\chi_0(s,x)=\chi_0(s,\psi(x))$.
Moreover, the kinetic equation \eqref{eq:hatchieps_kin} ensures that $\rho_\e$ is weakly differentiable, with
\begin{align*}
\partial_{\hat x_2}\rho_\e (\hat x_2)
&
=\int_{\T\times [-1,1]}
\!\!\!
\zeta(s,\hat x_1)\, \widehat\Theta_\e(s,\hat x)\,ds\, d\hat x_1 
\\
&\quad
+\int_{\T\times [-1,1]}
\!\!\!
\cos(s)(\hat\chi_\e(s,\hat x)-\mathrm{tr}(\hat\chi_*)(s,\hat x))\partial_{\hat x_1}\zeta(s,\hat x_1)\, ds d\hat x_1\,,
\end{align*}
hence
\begin{align*}
\|\rho_\e\|_{L^\infty(0,1)}
&
\lesssim \int_0^1 |\rho_\e|\,d \hat x_2 +\int_0^1 |\partial_{\hat x_2}\rho_\e|\, d\hat x_2
\\
&
\lesssim 
\|\zeta\|_\infty \|\hat a_\e -\hat a_*\|_{L^1(\T\times (B_2\cap H_+))}
\\
&
\quad
+\|\zeta\|_\infty \|\hat a_*\|_{\infty}
\|\hat\chi_\e-\hat \chi_*\|_{L^1(\T\times (B_2\cap H_+))} 
\\
&
\quad
+\|\zeta\|_\infty \|\widehat\Theta_\e\|_{L^1(\T\times (B_2\cap H_+))}
\\
&
\quad
+\|\partial_{\hat x_1}\zeta\|_\infty
\|\hat\chi_\e-\hat \chi_*\|_{L^1(\T\times (B_2\cap H_+))} \,.
\end{align*}
We infer that $\rho_\e\to 0$ uniformly on $(0,1)$ as $\e\to 0$.
Given the above expression of $\rho_\e(0^+)$ and the uniform convergence $\hat a_\e\to\hat a_*$, this implies
\begin{align*}
\lim_{\e\to 0}\int_{\T\times [-1,1]}\hat a_*(s) (\tilde\chi_0(s,x_*+\e\hat x_1)-\mathrm{tr}(\hat\chi_*)(s,\hat x_1))\zeta(s,\hat x_1)\, ds d\hat x_1 =0\,,
\end{align*}
for all $\zeta\in C_c^1(\T\times (-1,1))$.
Recall that this is valid for any $x_*\in V\cap H_+$.
The blow-up function $\hat\chi_*$ depends on $x_*$ and on the sequence $\e\to 0$, and its strong trace $\mathrm{tr}(\hat\chi_*)$ takes values in $\lbrace 0,1\rbrace$.
Further, for almost every choice of $x_*\in V\cap H_+$ we 
have
\begin{align*}
\int_{\T\times [-1,1]}|\tilde\chi_0(s,x_*+\e\hat x_1)-\tilde\chi_0(s,x_*)|\, ds d\hat x_1 \to 0\,,
\end{align*}
as $\e\to 0$. 
(This can be interpreted as almost every $x_*$ being a Lebesgue point of $x\mapsto \chi_0(\cdot,x)\in L^1(\T)$, and proved in the same way as the usual Lebesgue differentiation theorem. See also \cite[Lemma~3]{vasseur01} for a proof of this fact up to extracting a subsequence, which is sufficient for our purpose here.)
Combining this with the above and the fact that $\hat a_*(s)=\sin(s)-h(x_{*1})\cos(s) \neq 0$ for a.e. $s\in\T$, we conclude that
$\tilde\chi_0(s,x_*)\in\lbrace 0,1\rbrace$ for a.e. $(s,x_*)\in \T\times (V\cap\partial H_+)$,
hence $\chi_0=\tilde\chi_0\circ (\mathrm{id}_{\T}\otimes \psi(\cdot,0)^{-1})$ is an indicator function on $\T\times (\partial\Omega\cap U)$.
\end{proof}

\section{A symmetry question}\label{a:sym}

%Simulations seem to show that, even for not-so-small $\delta$, the minimizer is one-dimensional.

In the  non-thin periodic case,  
is it true that any minimizer $u$ of $E_\e$ in $[-1,1]^2$
subject to the boundary conditions $u(\pm 1,x_2)=e_\pm$ and $u(x_1,-1)=u(x_1,+1)$,
depends only on $x_1$ ?
In this Appendix we answer this question affirmatively.

Denote by 
$u^{\mathrm{1D}}_\e(x_1) =e^{i\varphi_\e^{\mathrm{1D}}(x_1)}$ 
the one-dimensional minimizer of the energy $E_\e^{\mathrm{1D}}$ restricted to maps which depend only on $x_1$, 
and $\mathfrak e_\e^{\mathrm{1D}}$ the value of its energy.
Denote also by $\varphi_\pm$ the boundary values of the phase $\varphi_\e^{\mathrm{1D}}$, which are such that $e_\pm =e^{i\varphi_\pm}$. 
Since
\begin{align*}
E_\e^{\mathrm{1D}}(e^{i\varphi})
&
=\int_{-1}^1
(\sin^2\varphi +\e \cos^2\varphi)(\varphi')^2\, dx
=\int_{-1}^1 \left([F_\e(\varphi)]'\right)^2\, dx,
\\
\text{where }
F_\e(t)
&
=\int_0^t \sqrt{\sin^2 s +\e\cos^2 s}\, ds,
\end{align*}
we have that $\varphi_\e^{\mathrm{1D}}$ solves
\begin{align*}
\frac{d^2}{dx_1^2}\big[ F_\e(\varphi_\e^{\mathrm{1D}})\big] =0\,.
\end{align*}
%where 
%\begin{align*}
%F_\e'(t)=\sqrt{\sin^2 t +\e \cos^2 t}\,.
%\end{align*}
It follows that the derivative of $F_\e(\varphi_\e^{\mathrm{1D}})$ with respect to $x_1$ is constant, and 
integrating on $(-1,1)$ we find the value of this constant:
\begin{align*}
\frac{d}{dx_1}\big[ F_\e(\varphi_\e^{\mathrm{1D}})\big] =\frac{F_\e(\varphi_+)-F_\e(\varphi_-)}{2}\,.
\end{align*}
Hence the value of the minimal 1D energy is
\begin{align*}
\mathfrak e_\e^{\mathrm{1D}} &
=  \int_{-1}^1 \left(\frac{d}{dx_1}\big[ F_\e(\varphi_\e^{1D})\big]\right)^2\, dx_1
=\frac 12 \big(F_\e(\varphi_+)-F_\e(\varphi_-)\big)^2\,,
\end{align*}
and $\varphi_\e^{\mathrm{1D}}$ solves
\begin{align*}
\frac{d\varphi_\e^{\mathrm{1D}}}{dx_1}
=\frac{\tau\sqrt{\mathfrak e_\e^{\mathrm{1D}}/2}}{F_\e'(\varphi_\e^{\mathrm{1D}})}\,,
\qquad\tau =\sign(\varphi_+-\varphi_-)\,.
\end{align*}
Guided by this equation,
we use the fact that 
$F_\e'(\varphi)^2=\sin^2\varphi +\e\cos^2\varphi$
 to rewrite the energy of any two-dimensional configuration $u=e^{i\varphi}$ with $\varphi(\pm 1,x_2)=\varphi_\pm$, as
\begin{align*}
2 E_\e(e^{i\varphi})
%&
%= \iint
%(-\sin\varphi \,\partial_1\varphi +\cos\varphi\,\partial_2\varphi)^2
%+
%\e (\cos\varphi\,\partial_1\varphi 
%+\sin\varphi\,\partial_2\varphi)^2
%\, dx
%\\
&
= \iint 
\Big(-\sin\varphi\, 
\Big(\partial_1\varphi -
\frac{\tau\sqrt{\mathfrak e_{\e}^{\mathrm{1D}}/2}}{F_\e'(\varphi)} 
+
\frac{\tau\sqrt{\mathfrak e_{\e}^{\mathrm{1D}}/2}}{F_\e'(\varphi)} 
\Big)
+\cos\varphi\,\partial_2\varphi
\Big)^2
\\
&\hspace{3em}
+\e
\Big(\cos\varphi\,
\Big(\partial_1\varphi -
\frac{\tau\sqrt{\mathfrak e_{\e}^{\mathrm{1D}}/2}}
{F_\e'(\varphi)}
+
\frac{\tau\sqrt{\mathfrak e_{\e}^{\mathrm{1D}}/2}}
{F_\e'(\varphi)}
\Big)
+\sin\varphi\,\partial_2\varphi
\Big)^2\, dx
\\
&
=
\iint
F_\e'(\varphi)^2
\Big(\partial_1\varphi -
\frac{\tau\sqrt{\mathfrak e_{\e}^{\mathrm{1D}}/2}}
{F_\e'(\varphi)}
+
\frac{\tau\sqrt{\mathfrak e_{\e}^{\mathrm{1D}}/2}}
{F_\e'(\varphi)}
\Big)^2
\\
&
\hspace{3em}
+(\cos^2\varphi +\e\sin^2\varphi)(\partial_2\varphi)^2
\\
&
\hspace{3em}
+2(\e-1)\cos\varphi\,
\sin\varphi\,
\Big(\partial_1\varphi -
\frac{\tau\sqrt{\mathfrak e_{\e}^{\mathrm{1D}}/2}}
{F_\e'(\varphi)}
+
\frac{\tau\sqrt{\mathfrak e_{\e}^{\mathrm{1D}}/2}}
{F_\e'(\varphi)}
\Big)\partial_2\varphi
\,\, dx\,.
\end{align*}
Expanding the integrand gives
\begin{align*}
&
2\big(E_\e(e^{i\varphi})-\mathfrak e_\e^{\mathrm{1D}}\big)
\\
&
=
\iint
F_\e'(\varphi)^2\Big(\partial_1\varphi 
-\frac{\tau\sqrt{\mathfrak e_\e^{\mathrm{1D}}/2}}{ F_\e'(\varphi)}\Big)^2
+(\cos^2\varphi +\e\sin^2\varphi)(\partial_2\varphi)^2
\\
&
\hspace{3em}
+2(\e-1)\cos\varphi\,
\sin\varphi\,
\Big(\partial_1\varphi -
\frac{\tau\sqrt{\mathfrak e_{\e}^{\mathrm{1D}}/2}}
{F_\e'(\varphi)}
\Big)\partial_2\varphi
\\
&\hspace{3em}
+
2\tau \sqrt{\mathfrak e_\e^{\mathrm{1D}}/2}\Big( F_\e'(\varphi)\partial_1\varphi 
-\tau\sqrt{\mathfrak e_\e^{\mathrm{1D}}/2}\Big)
\\
&
\hspace{3em}
+
2(\e - 1) 
\tau\sqrt{\mathfrak e_{\e}^{\mathrm{1D}}/2}
\frac{\cos\varphi\,\sin\varphi}{F_\e'(\varphi)}\partial_2\varphi
\,\, dx\,.
\end{align*}
The last line integrates to zero by $x_2$-periodicity,
and so does the next-to-last line, because
\begin{align*}
\int_{-1}^1 F_\e'(\varphi)\partial_1\varphi\, dx_1
=\int_{-1}^1 \partial_1 [F_\e(\varphi)]\, dx_1
=F_\e(\varphi_+)-F_\e(\varphi_-) =\tau\sqrt{2\mathfrak e_\e^{\mathrm{1D}}}\,.
\end{align*}
Thus we obtain, recalling also that $F_\e'(t)^2=\sin^2 t +\e\cos^2 t$,
\begin{align*}
&
2\big(E_\e(e^{i\varphi})-\mathfrak e_\e^{\mathrm{1D}}\big)
\\
&
=
\iint
(\sin^2\varphi +\e\cos^2\varphi)
\Big( \partial_1\varphi 
-\frac{\tau\sqrt{\mathfrak e_\e^{\mathrm{1D}}/2}}{F_\e'(\varphi)}\Big)^2
+(\cos^2\varphi +\e\sin^2\varphi)(\partial_2\varphi)^2
\\
&
\hspace{3em}
+2(\e-1)\cos\varphi\,
\sin\varphi\,
\Big(\partial_1\varphi -
\frac{\tau\sqrt{\mathfrak e_{\e}^{\mathrm{1D}}/2}}
{F_\e'(\varphi)}
\Big)\partial_2\varphi
\,\, dx
\\
&
=
\iint
\Big(\frac{1+\e}{2}-\frac{1-\e}{2}\cos(2\varphi)\Big)
\Big( \partial_1\varphi 
-\frac{\tau\sqrt{\mathfrak e_\e^{\mathrm{1D}}/2}}{F_\e'(\varphi)}\Big)^2
\\
&
\hspace{3em}
+\Big(\frac{1+\e}{2}+\frac{1-\e}{2}\cos(2\varphi)\Big)(\partial_2\varphi)^2
\\
&
\hspace{3em}
+(\e-1)
\sin(2\varphi)\,
\Big(\partial_1\varphi -
\frac{\tau\sqrt{\mathfrak e_{\e}^{\mathrm{1D}}/2}}
{F_\e'(\varphi)}
\Big)\partial_2\varphi
\,\, dx
\end{align*}
For any $C,S\in\R$ such that $C^2+S^2=1$, 
the quadratic form on $\R^2$ given by
\begin{align*}
q(X,Y)
&
=
((1+\e)-(1-\e)C)X^2
\\
&\quad
+((1+\e)+(1-\e)C)Y^2
\\
&\quad
+2(\e-1)SXY\,,
\end{align*}
has determinant $\det( q) = 4\e$ and trace $\tr (q) = 2(1+\e)$, hence it satisfies
\begin{align*}
q(X,Y)\geq 2\min(1,\e) (X^2+Y^2)\,.
\end{align*}
Applying this to the integrand in the above expression of the energy difference,
we deduce that
\begin{align*}
&
E_\e(e^{i\varphi})-\mathfrak e_\e^{\mathrm{1D}}
%\\
%&
\geq \frac{\min(1,\e)}{2}
\iint
(\partial_2\varphi)^2
+
\bigg( \partial_1\varphi 
-\frac{\tau\sqrt{\mathfrak e_\e^{\mathrm{1D}}/2}}{F_\e'(\varphi)}\bigg)^2
\, dx\,,
\end{align*}
for any  $H^1$ function $\varphi$ 
satisfying the boundary conditions 
$\varphi(\pm 1,x_2)=\varphi_\pm$
and
$\varphi(x_1,-1)=\varphi(x_1,+1)$. The assertion made in this appendix is thus proved.

\end{appendices}

\small

\bibliographystyle{acm}
\bibliography{high_anis}

\begin{thebibliography}{10}

\bibitem{ADM99}
{\sc Ambrosio, L., De~Lellis, C., and Mantegazza, C.}
\newblock Line energies for gradient vector fields in the plane.
\newblock {\em Calc. Var. Partial Differential Equations 9}, 4 (1999),
  327--255.

\bibitem{braides}
{\sc Braides, A.}
\newblock {\em {{\(\Gamma\)}}-convergence for beginners}, vol.~22 of {\em Oxf.
  Lect. Ser. Math. Appl.}
\newblock Oxford: Oxford University Press, 2002.

\bibitem{BM21}
{\sc Brezis, H., and Mironescu, P.}
\newblock {\em Sobolev maps to the circle. {From} the perspective of analysis,
  geometry, and topology}, vol.~96 of {\em Prog. Nonlinear Differ. Equ. Appl.}
\newblock New York, NY: Birkh{\"a}user, 2021.

\bibitem{BN95}
{\sc Br{\'e}zis, H., and Nirenberg, L.}
\newblock Degree theory of {BMO}. {I}: {Compact} manifolds without boundaries.
\newblock {\em Sel. Math., New Ser. 1}, 2 (1995), 197--263.

\bibitem{BCS25}
{\sc Bronsard, L., Colinet, A., and Stantejsky, D.}
\newblock \emph{A priori} {{\(L^\infty\)}}-bound for {Ginzburg}-{Landau} energy
  minimizers with divergence penalization.
\newblock {\em Commun. Partial Differ. Equations 50}, 4 (2025), 542--569.

\bibitem{BCSvB}
{\sc Bronsard, L., Colinet, A., Stantejsky, D., and van Brussel, L.}
\newblock {A priori estimates and $\eta-$compactness for anisotropic
  Ginzburg-Landau minimizers with tangential anchoring}.
\newblock In preparation.

\bibitem{BMMPS}
{\sc Bronsard, L., Merlet, B., Monteil, A., Pegon, M., and Stantejsky, D.}
\newblock {On a Ginzburg-Landau energy with heavily penalized divergence and
  tangential anchoring}.
\newblock In preparation.

\bibitem{CF99}
{\sc Chen, G.-Q., and Frid, H.}
\newblock Divergence-measure fields and hyperbolic conservation laws.
\newblock {\em Arch. Ration. Mech. Anal. 147}, 2 (1999), 89--118.

\bibitem{CKP13}
{\sc Colbert-Kelly, S., and Phillips, D.}
\newblock Analysis of a {G}inzburg-{L}andau type energy model for smectic
  {$C^\ast$} liquid crystals with defects.
\newblock {\em Ann. Inst. H. Poincar\'{e} Anal. Non Lin\'{e}aire 30}, 6 (2013),
  1009--1026.

\bibitem{CET94}
{\sc Constantin, P., E, W., and Titi, E.~S.}
\newblock Onsager's conjecture on the energy conservation for solutions of
  {Euler}'s equation.
\newblock {\em Commun. Math. Phys. 165}, 1 (1994), 207--209.

\bibitem{DM97}
{\sc Dacorogna, B., and Marcellini, P.}
\newblock General existence theorems for {H}amilton-{J}acobi equations in the
  scalar and vectorial cases.
\newblock {\em Acta Math. 178}, 1 (1997), 1--37.

\bibitem{degennes}
{\sc De~Gennes, P.-G., and Prost, J.}
\newblock {\em The physics of liquid crystals}.
\newblock Oxford university press, 1993.

\bibitem{DLI15}
{\sc De~Lellis, C., and Ignat, R.}
\newblock A regularizing property of the {{\(2D\)}}-eikonal equation.
\newblock {\em Commun. Partial Differ. Equations 40}, 8 (2015), 1543--1557.

\bibitem{DKMO01}
{\sc DeSimone, A., Kohn, Robert V.and~M{\"u}ller, S., and Otto, F.}
\newblock A compactness result in the gradient theory of phase transitions.
\newblock {\em Proc. R. Soc. Edinb., Sect. A, Math. 131}, 4 (2001), 833--844.

\bibitem{DNPV12}
{\sc Di~Nezza, E., Palatucci, G., and Valdinoci, E.}
\newblock Hitchhiker's guide to the fractional {Sobolev} spaces.
\newblock {\em Bull. Sci. Math. 136}, 5 (2012), 521--573.

\bibitem{evans90}
{\sc Evans, L.~C.}
\newblock {\em Weak convergence methods for nonlinear partial differential
  equations. {Expository} lectures from the {CBMS} regional conference held at
  {Loyola} {University} of {Chicago}, {June} 27-{July} 1, 1988}, vol.~74 of
  {\em Reg. Conf. Ser. Math.}
\newblock Providence, RI: American Mathematical Society, 1990.

\bibitem{GL20}
{\sc Ghiraldin, F., and Lamy, X.}
\newblock Optimal {B}esov differentiability for entropy solutions of the
  eikonal equation.
\newblock {\em Comm. Pure Appl. Math. 73}, 2 (2020), 317--349.

\bibitem{GJO15}
{\sc Goldman, M., Josien, M., and Otto, F.}
\newblock New bounds for the inhomogenous {Burgers} and the
  {Kuramoto}-{Sivashinsky} equations.
\newblock {\em Commun. Partial Differ. Equations 40}, 12 (2015), 2237--2265.

\bibitem{GMPS24}
{\sc Goldman, M., Merlet, B., Pegon, M., and Serfaty, S.}
\newblock Compactness and structure of zero-states for unoriented
  {Aviles}-{Giga} functionals.
\newblock {\em J. Inst. Math. Jussieu 23}, 2 (2024), 941--982.

\bibitem{GMS24}
{\sc Golovaty, D., Mironescu, P., and Sternberg, P.~J.}
\newblock Towards an asymptotic analysis of the anisotropic {Ginzburg}-{Landau}
  model.
\newblock {\em arXiv:2407.00512\/} (2024).

\bibitem{GNS21}
{\sc Golovaty, D., Novack, M., and Sternberg, P.}
\newblock A one-dimensional variational problem for cholesteric liquid crystals
  with disparate elastic constants.
\newblock {\em J. Differ. Equations 286\/} (2021), 785--820.

\bibitem{GNSV20}
{\sc Golovaty, D., Novack, M., Sternberg, P., and Venkatraman, R.}
\newblock A model problem for nematic-isotropic transitions with highly
  disparate elastic constants.
\newblock {\em Arch. Ration. Mech. Anal. 236}, 3 (2020), 1739--1805.

\bibitem{GSV19}
{\sc Golovaty, D., Sternberg, P., and Venkatraman, R.}
\newblock A {G}inzburg-{L}andau-type problem for highly anisotropic nematic
  liquid crystals.
\newblock {\em SIAM J. Math. Anal. 51}, 1 (2019), 276--320.

\bibitem{golse09}
{\sc Golse, F.}
\newblock Nonlinear regularizing effect for conservation laws.
\newblock In {\em Hyperbolic problems. Theory, numerics and applications.
  Plenary and invited talks. Proceedings of the 12th international conference
  on hyperbolic problems, June 9--13, 2008}. Providence, RI: American
  Mathematical Society (AMS), 2009, pp.~73--92.

\bibitem{GP13}
{\sc Golse, F., and Perthame, B.}
\newblock Optimal regularizing effect for scalar conservation laws.
\newblock {\em Rev. Mat. Iberoam. 29}, 4 (2013), 1477--1504.

\bibitem{GRS24}
{\sc Guerra, A., Rai\c{t}\u{a}, B., and Schrecker, M.}
\newblock Compensation phenomena for concentration effects via nonlinear
  elliptic estimates.
\newblock {\em Ars Inven. Anal.\/} (2024), Paper No. 1, 56.

\bibitem{HKL86}
{\sc Hardt, R., Kinderlehrer, D., and Lin, F.-H.}
\newblock Existence and partial regularity of static liquid crystal
  configurations.
\newblock {\em Comm. Math. Phys. 105}, 4 (1986), 547--570.

\bibitem{HKL88}
{\sc Hardt, R., Kinderlehrer, D., and Lin, F.-H.}
\newblock Stable defects of minimizers of constrained variational principles.
\newblock {\em Ann. Inst. H. Poincar\'e Anal. Non Lin\'eaire 5}, 4 (1988),
  297--322.

\bibitem{IO08}
{\sc Ignat, R., and Otto, F.}
\newblock A compactness result in thin-film micromagnetics and the optimality
  of the {N{\'e}el} wall.
\newblock {\em J. Eur. Math. Soc. (JEMS) 10}, 4 (2008), 909--956.

\bibitem{JP01}
{\sc Jabin, P.-E., and Perthame, B.}
\newblock Compactness in {Ginzburg}-{Landau} energy by kinetic averaging.
\newblock {\em Commun. Pure Appl. Math. 54}, 9 (2001), 1096--1109.

\bibitem{LS22}
{\sc Leonardi, G.~P., and Saracco, G.}
\newblock Rigidity and trace properties of divergence-measure vector fields.
\newblock {\em Adv. Calc. Var. 15}, 1 (2022), 133--149.

\bibitem{leoni}
{\sc Leoni, G.}
\newblock {\em A first course in {Sobolev} spaces}, 2nd edition~ed., vol.~181
  of {\em Grad. Stud. Math.}
\newblock Providence, RI: American Mathematical Society (AMS), 2017.

\bibitem{LPT94}
{\sc Lions, P.~L., Perthame, B., and Tadmor, E.}
\newblock A kinetic formulation of multidimensional scalar conservation laws
  and related equations.
\newblock {\em J. Am. Math. Soc. 7}, 1 (1994), 169--191.

\bibitem{LP23}
{\sc Lorent, A., and Peng, G.}
\newblock Factorization for entropy production of the eikonal equation and
  regularity.
\newblock {\em Indiana Univ. Math. J. 72}, 3 (2023), 1055--1105.

\bibitem{murat78}
{\sc Murat, F.}
\newblock Compacité par compensation.
\newblock {\em Ann. Sc. Norm. Super. Pisa, Cl. Sci., IV. Ser. 5\/} (1978),
  489--507.

\bibitem{JOP02}
{\sc Otto, F., Jabin, P.-E., and Perthame, B.}
\newblock Line-energy {Ginzburg}-{Landau} models: zero-energy states.
\newblock {\em Ann. Sc. Norm. Super. Pisa, Cl. Sci. (5) 1}, 1 (2002), 187--202.

\bibitem{perthame02}
{\sc Perthame, B.}
\newblock {\em Kinetic formulation of conservation laws}, vol.~21 of {\em Oxf.
  Lect. Ser. Math. Appl.}
\newblock Oxford: Oxford University Press, 2002.

\bibitem{PS98}
{\sc Perthame, B., and Souganidis, P.~E.}
\newblock A limiting case for velocity averaging.
\newblock {\em Ann. Sci. {\'E}c. Norm. Sup{\'e}r. (4) 31}, 4 (1998), 591--598.

\bibitem{RS03}
{\sc Rivi{\`e}re, T., and Serfaty, S.}
\newblock Compactness, kinetic formulation, and entropies for a problem related
  to micromagnetics.
\newblock {\em Commun. Partial Differ. Equations 28}, 1-2 (2003), 249--269.

\bibitem{tartar79}
{\sc Tartar, L.}
\newblock Compensated compactness and applications to partial differential
  equations.
\newblock Nonlinear analysis and mechanics: {Heriot}-{Watt} {Symp}., {Vol}. 4,
  {Edinburgh} 1979, {Res}. {Notes} {Math}. 39, 136-212 (1979)., 1979.

\bibitem{triebel83}
{\sc Triebel, H.}
\newblock {\em Theory of function spaces}, vol.~78 of {\em Monogr. Math.,
  Basel}.
\newblock Birkh{\"a}user, Cham, 1983.

\bibitem{vasseur01}
{\sc Vasseur, A.}
\newblock Strong traces for solutions of multidimensional scalar conservation
  laws.
\newblock {\em Arch. Ration. Mech. Anal. 160}, 3 (2001), 181--193.

\bibitem{Virga94}
{\sc Virga, E.~G.}
\newblock {\em Variational theories for liquid crystals}.
\newblock Chapman and Hall/CRC, 1994.

\end{thebibliography}

\end{document}